\documentclass[reqno]{amsart}
\usepackage{amssymb}
\usepackage{enumitem}
\usepackage{bbm}
\usepackage[usenames, dvipsnames]{color}

\usepackage{hyperref}
\usepackage{amsmath}

\usepackage{verbatim}
\usepackage{todonotes}
\usepackage{pxfonts}
\usepackage{tikz}
\usetikzlibrary{arrows, automata}

\allowdisplaybreaks[4]

\theoremstyle{plain}
\newtheorem{theorem}{Theorem}[section]
\newtheorem{lemma}[theorem]{Lemma}

\newtheorem{proposition}[theorem]{Proposition}

\theoremstyle{definition}

\newtheorem{assumption}[theorem]{Assumption}

\theoremstyle{remark}
\newtheorem{remark}[theorem]{Remark}

\newcommand{\Div}{\operatorname{div}}

\numberwithin{equation}{section}

\newcommand{\bR}{\mathbb{R}}

\newcommand\cD{\mathcal{D}}

\makeatletter
\def\dashint{\operatorname%
{\,\,\text{\bf--}\kern-.98em\DOTSI\intop\ilimits@\!\!}}
\makeatother

\begin{document}
	
\title[Higher regularity of Stokes systems]{On higher regularity of Stokes systems with piecewise H\"{o}lder continuous coefficients}

\author[H. Dong]{Hongjie Dong}
\address[H. Dong]{Division of Applied Mathematics, Brown University, 182 George Street, Providence, RI 02912, USA}
\email{Hongjie\_Dong@brown.edu}
\thanks{H. Dong was partially supported by Simons Fellows Award 007638 and the NSF under agreement DMS-2055244.}

\author[H. Li]{Haigang Li}
\address[H. Li]{School of Mathematical Sciences, Beijing Normal University, Laboratory of Mathematics and Complex Systems, Ministry of Education, Beijing 100875, China.}
\email{hgli@bnu.edu.cn}
\thanks{H. Li was partially supported by NSF of China (11971061).}

\author[L. Xu]{Longjuan Xu}
\address[L. Xu]{Academy for Multidisciplinary Studies, Capital Normal University, Beijing 100048, China.}
\email{longjuanxu@cnu.edu.cn}
\thanks{L. Xu was partially supported by NSF of China (12301141).}

\begin{abstract}
In this paper, we consider higher regularity of a weak solution $({\bf u},p)$ to stationary Stokes systems with variable coefficients. Under the assumptions that coefficients and data are piecewise $C^{s,\delta}$ in a  bounded domain consisting of a finite number of  subdomains with interfacial boundaries in $C^{s+1,\mu}$, where $s$ is a positive integer, $\delta\in (0,1)$, and $\mu\in (0,1]$, we show that $D{\bf u}$ and $p$ are piecewise $C^{s,\delta_{\mu}}$, where $\delta_{\mu}=\min\big\{\frac{1}{2},\mu,\delta\big\}$. Our result is new even in the 2D case with piecewise constant coefficients.
\end{abstract}

\maketitle

\section{Introduction and main results}
Stokes systems with variable coefficients have been studied extensively in the literature.
See, for instance, the pioneer work of Giaquinta and Modica \cite{gm82}. Such type of Stokes systems can be used to model the motion of inhomogeneous fluid with density dependent viscosity  \cite{ls1975,l1996,agz2011}.
In this paper, we study stationary Stokes systems with piecewise smooth coefficients
\begin{align}\label{stokes}
\begin{cases}
D_\alpha (A^{\alpha\beta}D_\beta {\bf u})+Dp=D_{\alpha}{\bf f}^{\alpha}\\
\Div {\bf u}=g
\end{cases}\,\,\mbox{in }~\cD
\end{align}
where ${\bf u}=(u^1,\ldots, u^d)^{\top}$ and ${\bf f}^{\alpha}=(f_1^\alpha,\ldots,f_d^\alpha)^{\top}$, $d\geq2$, and we used  the Einstein summation convention over repeated indices.  We assume that the bounded domain $\cD$ in $\mathbb R^d$ contains a finite number of disjoint subdomains $\cD_j$, $j=1,\dots,M$, and the coefficients and the data may have jump across the boundaries of the subdomains. By approximation, we may assume that any point $x\in\cD$ belongs to the boundaries of at most two of the $\cD_{j}$'s. With these assumptions, the Stokes systems \eqref{stokes}  is connected to the study of composite materials with closely spaced interfacial boundaries (see, for instance, \cite{mnmv1987,jg2015}), as well as the study of the motion of two fluids with interfacial boundaries \cite{cd2019,dk2018,dk2019,kw2018,kmw2021}.

This problem is also stimulated by the study of regularity of weak solutions for equations with rough coefficients. There have been significant developments on the regularity theory for partial differential equations and systems with coefficients which satisfy some proper piecewise continuous conditions.  We shall begin by reviewing the literature for results on gradient estimates in such a setting from the past two decades.  Bonnetier and Vogelius first \cite{bv2000} considered divergence form second-order elliptic equations with piecewise constant coefficients:
\begin{equation}\label{homoscalar}
D_\alpha(a(x)D_\alpha u)=0\quad\mbox{in}~\cD,
\end{equation}
where $a(x)$ is given by
\begin{align*}
a(x)=a_0\mathbbm{1}_{\cD_1\cup\cD_2}+\mathbbm{1}_{\cD\setminus(\cD_1\cup\cD_2)},
\end{align*}
with $0<a_0<\infty$ and $\mathbbm{1}_{\bullet}$ is the indicator function.
They proved that the gradient of the solution is bounded when the subdomains are circular touching fibers of comparable
radii. Li and Vogelius \cite{lv2000} studied general elliptic equations in divergence form:
\begin{equation*}
D_\alpha(A^{\alpha\beta}D_\beta u)=D_\alpha f^\alpha\quad\mbox{in}~\cD,
\end{equation*}
where the coefficients $A^{\alpha\beta}$ and the data $f^\alpha$ are $C^{\delta}$ ($\delta\in(0,1)$) up to the boundary in each subdomain with $C^{1,\mu}$ boundary, $\mu\in(0,1]$, but may have jump discontinuities across the boundaries of the subdomains. They established global Lipschitz and piecewise $C^{1,\delta'}$ estimates of the solution with $\delta'\in(0,\min\{\delta,\frac{\mu}{d(\mu+1)}\}]$. This result was extended to elliptic systems under the same conditions by Li and Nirenberg \cite{ln2003} and the range of $\delta'$ was improved to  $\delta'\in(0,\min\{\delta,\frac{\mu}{2(\mu+1)}\}]$.  Dong and Xu \cite{dx2019} further relaxed the range to $\delta'\in(0,\min\{\delta,\frac{\mu}{\mu+1}\}]$ by using a completely different argument from \cite{lv2000,ln2003}. Notably, the estimates in \cite{lv2000,ln2003,dx2019} are independent of the distances between subdomains. For more related results, we refer the reader to \cite{ckvc86,cf2012,d2012,xb13,z2021} and the references therein. The estimates were extended to the case of parabolic equations and systems with piecewise continuous coefficients \cite{fknn13,ll2017,dx2021}, and stationary Stokes systems with piecewise Dini mean oscillation coefficients \cite{cdx2022}.

Now let us discuss the topic of the higher regularity for solutions to partial differential equations and systems with piecewise smooth coefficients. Significant progresses have been made on the second-order elliptic equations \eqref{homoscalar} with piecewise constant coefficients.
By using conformal mappings, Li and Vogelius \cite{lv2000} proved that the solutions to \eqref{homoscalar} are piecewise smooth up to interfacial boundaries, when the subdomains $\cD_1$ and $\cD_2$ are two touching unit disks in $\mathbb R^2$, and $\cD$ is a disk $B_{R_0}$ with sufficiently large $R_0$. Dong and Zhang \cite{dz2016} removed the requirement that $R_0$ being sufficiently large with the help of the construction of Green's function. Dong and Li \cite{dl2019} then applied the Green function method to obtain higher derivative estimates by demonstrating the explicit dependence of the coefficients and the distance between interfacial boundaries of inclusions. Related results about higher derivative estimates with circular inclusions were investigated in \cite{jk2023,dy2023}. It is worth noting that in all these work, the dimension is always assumed to be two and the inclusions are circular. To the best of our knowledge, there is no corresponding result available for Stokes systems.

Recently, Dong and Xu \cite{dx2022} tackled more general divergence form parabolic systems in any dimensions with piecewise H\"{o}lder continuous coefficients and data in a bounded domain consisting of a finite number of cylindrical subdomains. By using a completely different method from those in  \cite{lv2000,dz2016, dl2019, jk2023,dy2023}, they established piecewise higher derivative estimates for weak solutions to such parabolic systems, and the estimates are independent of the distance between the interfaces. This result also implies piecewise higher regularity for the corresponding elliptic systems, addressing the open question proposed in \cite{lv2000}.

In this paper,  we study higher regularity for solutions to the Stokes system \eqref{stokes}, closely following the scheme in \cite{dx2022}. However, the presence of the pressure term $p$ introduces added difficulties in the proofs below.

To state our main result precisely, we first give the following assumption imposed on the domain $\cD$.
\begin{assumption}\label{assumpdomain}
The bounded domain $\cD$ in $\mathbb R^d$ contains $M$ disjoint subdomains $\cD_{j},j=1,\ldots,M$, and the interfacial boundaries are $C^{s+1,\mu}$, where $s\in\mathbb N$ and $\mu\in(0,1]$. We also assume that any point $x\in\cD$ belongs to the boundaries of at most two of the $\cD_{j}$'s.
\end{assumption}

For $0<\delta<1$, we denote the $C^{\delta}$ H\"older semi-norm by
$$
[u]_{C^{\delta}(\cD)}:=\sup_{\substack{x,y\in \cD \\ x\neq y}} \frac{|u(x)-u(y)|}{|x-y|^\delta},
$$
and the $C^{\delta}$ norm by
$$|u|_{\delta;\cD}:=[u]_{C^{\delta}(\cD)}+|u|_{0;\cD},\quad \text{where}\,\,|u|_{0;\cD}=\sup_{\cD}|u|.$$
By $C^\delta(\cD)$ we denote the set for all bounded measurable functions $u$ satisfying $[u]_{C^{\delta}(\cD)}<\infty$. The function spaces $C^{s,\delta}(\cD),s\in\mathbb N,$ are defined accordingly. For $\varepsilon>0$ small, we set
$$\cD_{\varepsilon}:=\{x\in \cD: \mbox{dist}(x,\partial \cD)>\varepsilon\}.$$

\begin{assumption}\label{assump}
The coefficients $A^{\alpha\beta}$ are bounded and satisfy the strong ellipticity condition, that is, there exists $\nu\in (0,1)$ such that
\begin{equation*}
|A^{\alpha\beta}(x)|\le \nu^{-1}, \quad \sum_{\alpha,\beta=1}^d A^{\alpha\beta}(x)\xi_\beta\cdot \xi_\alpha\ge \nu \sum_{\alpha=1}^d |\xi_\alpha|^2
\end{equation*}
for any $x\in \bR^d$ and $\xi_\alpha\in \bR^d$, $\alpha\in \{1,\ldots, d\}$. Moreover, $A^{\alpha\beta}$,  ${\bf f}^{\alpha}$, and $g$ are assumed to be of class $C^{s,\delta}(\cD_{\varepsilon}\cap\overline{{\cD}_{j}}),j=1,\ldots,M$, where $s\in\mathbb N$ and $\delta\in(0,1)$.
\end{assumption}

Here is our main result.
\begin{theorem}\label{Mainthm}
Let $\varepsilon\in (0,1)$ and $q\in(1,\infty)$.  Assume that $\cD$ satisfies Assumption \ref{assumpdomain}, and $A^{\alpha\beta}$, ${\bf f}^{\alpha}$, and $g$ satisfy Assumption \ref{assump}. Let $({\bf u},p)\in W^{1,q}(\cD)^d\times L^q(\cD)$ be a weak solution to \eqref{stokes} in $\cD$. Then $({\bf u},p)\in C^{s+1,\delta_{\mu}}( \cD_{\varepsilon}\cap\overline{{\cD}_{j_0}})^d\times C^{s,\delta_{\mu}}( \cD_{\varepsilon}\cap\overline{{\cD}_{j_0}})$ and  it holds that
\begin{align*}
|{\bf u}|_{s+1,\delta_{\mu};\cD_{\varepsilon}\cap\overline{{\cD}_{j_0}}}+|p|_{s,\delta_{\mu};\cD_{\varepsilon}\cap\overline{{\cD}_{j_0}}}\leq N\Big(\|D{\bf u}\|_{L^{1}(\cD)}+\|p\|_{L^{1}(\cD)}+\sum_{j=1}^{M}|{\bf f}^\alpha|_{s,\delta;\overline{\cD_{j}}}+\sum_{j=1}^{M}|g|_{s,\delta;\overline{\cD_{j}}}\Big),
\end{align*}
where $j_0=1,\ldots,M$, $\delta_{\mu}=\min\big\{\frac{1}{2},\mu,\delta\big\}$, $N$ depends on $d$, $M$, $q$, $\nu$, $\varepsilon$, $|A|_{s,\delta;\overline{\cD_{j}}}$, and the $C^{s+1,\mu}$ characteristic of $\cD_{j}$.
\end{theorem}

\begin{remark}
The piecewise H\"{o}lder-regularity of $(D{\bf u},p)$ for $s=0$ was proved in \cite{cdx2022} with $\delta_{\mu}=\min\{\delta,\frac{\mu}{\mu+1}\}$.
As mentioned in \cite[p. 3616]{cdx2022}, the results in Theorem \ref{Mainthm} can also be applied  to anisotropic Stokes systems in the form
\begin{align*}
\begin{cases}
\operatorname{div} (\tau \mathcal{S}{\bf u})+D p=D_{\alpha}{\bf f}^{\alpha}\\
\Div {\bf u}=g
\end{cases}\,\,\mbox{in }~\cD,
\end{align*}
where $\tau=\tau(x)$ is a  piecewise $C^{s,\delta}$ scalar function satisfying $\nu\le \tau\le \nu^{-1}$ and
$\mathcal{S}{\bf u}=\frac{1}{2}(D{\bf u}+(D{\bf u})^{\top})$
is the  rate of deformation tensor or strain tensor.
\end{remark}

The remainder of this paper is structured as follows: Section \ref{secpreliminaries} provides an overview of the notation, vector fields, and coordinate systems introduced in \cite{dx2022}, along with several auxiliary results. In Section \ref{newsystem}, we derive a new Stokes system for the case when $s=1$. Sections \ref{auxilemma} and \ref{bddestimate} contain the key components of the proof of Theorem \ref{Mainthm} with $s=1$. It is important to note that we encounter challenges due to the presence of the pressure term $p$, as exemplified in the proof of Lemma \ref{lemmaup} below. Finally, in Section \ref{prfprop}, we conclude the proof of Theorem \ref{Mainthm} with $s=1$ by utilizing the results from Sections \ref{auxilemma} and \ref{bddestimate}. In Section \ref{general}, we extend the proof to cover Theorem \ref{Mainthm} for general $s\geq2$.

\section{Preliminaries}\label{secpreliminaries}
In this section, we first review the notation, vector fields, and coordinate systems in \cite{dx2022}. Then we give some auxiliary lemmas which will be used in the proof of our results.

\subsection{Notation, vector fields, and coordinate systems}
We use $x=(x',x^d)$ to denote a generic point in the Euclidean space $\bR^d$, where $d\ge 2$ and $x'=(x^1,\ldots, x^{d-1})\in \bR^{d-1}$.
For $r>0$, we denote
$$B_{r}(x)=\{y\in\mathbb R^{d}: |y-x|<r\},\quad B'_{r}(x')=\{y'\in\mathbb R^{d-1}: |y'-x'|<r\}.$$
We often write $B_r$ and $B'_r$ for $B_r(0)$ and $B'_r(0)$, respectively.
For $q\in (0, \infty]$, we define
$$
L_0^q(\cD)=\{f\in L^q(\cD): (f)_{\cD}=0\},
$$
where $(f)_{\cD}$ is the average of $f$ over $\cD$:
$$
(f)_\cD=\fint_{\cD} f\, dx=\frac{1}{|\cD|}\int_{\cD} f \,dx.
$$
We denote by $W^{1,q}(\cD)$ the usual Sobolev space and by $W^{1,q}_0(\cD)$ the completion of $C^\infty_0(\cD)$ in $W^{1,q}(\cD)$, where $C^\infty_0(\cD)$ is the set of all infinitely differentiable functions with a compact support in $\cD$.

For simplicity, we take $\cD$ to be $B_1$. By suitable rotation and scaling, we may suppose that a finite number of subdomains lie in $B_{1}$ and that they can be represented by
$$x^{d}=h_{j}(x'),\quad x'\in B'_{1},~j=1,\ldots,m(<M),$$
where
\begin{equation*}
-1<h_{1}(x')<\dots<h_{m}(x')<1,
\end{equation*}
$h_{j}(x')\in C^{s+1,\mu}(B'_{1})$ with $s\in\mathbb N$. Set $h_{0}(x')=-1$ and $h_{m+1}(x')=1$. Then we have $m+1$ regions:
$$\cD_{j}:=\{x\in \cD: h_{j-1}(x')<x^{d}<h_{j}(x')\},\quad1\leq j\leq m+1.$$
The interfacial boundary is denoted by $\Gamma_j:=\{x^d=h_j(x')\}$, and the normal direction of $\Gamma_j$ is given by
\begin{equation}\label{normal}
{\bf n}_j:=(n_j^1,\ldots,n_j^d)=\frac{(-D_{x'}h_j(x'),1)^\top}{(1+|D_{x'}h_j(x')|^2)^{1/2}}\in \bR^{d},\quad j=1,\ldots,m.
\end{equation}

As in \cite[Section 2.3]{dx2022}, we fix a coordinate system such that $0\in\cD_{i_0}$ for some $i_0\in\{1,\ldots,m+1\}$ and the closest point on $\partial\cD_{i_0}$ is $x_{i_0}=(0',h_{i_0}(0'))$, and $\nabla_{x'}h_{i_0}(0')=0'$. 
In this coordinate system, we shall use $x=(x',x^d)$ and $D_{x}$ to denote the point and the derivatives, respectively.

The following vector field was introduced in \cite{dx2022}. For the completeness of the paper and reader's convenience, we review it here. For each $k=1,\ldots,d-1$, we define a vector field $\ell^{k,0}:\bR^{d}\to \bR^d$ near the center point $0$ of $B_1$ as follows: $\ell^{k,0}=(0,\ldots,0,1,0,\ldots,\ell_d^{k,0})$, where
$$
\ell^{k,0}_i=\delta_{ki},\quad i=1,\dots,d-1,
$$
$\delta_{ki}$ are Kronecker delta symbols, and
\begin{equation*}
\ell^{k,0}_d=
\begin{cases}
D_kh_m(x'),\quad&x^d\ge h_m,\\
\frac {x^d-h_{j-1}}{h_{j}-h_{j-1}}D_kh_j(x')+\frac {h_{j}-x^d}{h_{j}-h_{j-1}}
D_kh_{j-1}(x'),\quad&h_{j-1}\le x^d< h_j,~j=1,\dots,m,\\
D_kh_1(x'),\quad&x^d< h_1.
\end{cases}
\end{equation*}
Here, $D_{k}:=D_{x_k}$.
One can see that $\ell_d^{k,0}=D_kh_{j}(x')$ on $\Gamma_j$ and thus  $\ell^{k,0}$ is in a tangential direction. Moreover, it follows from $h_j\in C^{s+1,\mu}$ that $\ell^{k,0}$ is $C^{s,\mu}$ on $\Gamma_j$. Introduce the projection operator defined by
$$\mbox{proj}_{a}b=\frac{\langle a,b\rangle}{\langle a,a\rangle}a,$$
where $\langle a,b\rangle$ denotes the inner product of the vectors $a$ and $b$, and $\langle a,a\rangle=|a|^{2}$. By using the Gram-Schmidt process:
\begin{equation}\label{defell}
\begin{split}
\tilde\ell^{1}&=\ell^{1,0},
\quad\ell^1={\tilde\ell^{1}}/{|\tilde\ell^{1}|},\\
\tilde\ell^{2}&=\ell^{2,0}
-\mbox{proj}_{\ell^{1}}\ell^{2,0},
\quad\ell^2={\tilde\ell^{2}}/{|\tilde\ell^{2}|},\\
&\vdots\\
\tilde\ell^{d-1}&=\ell^{d-1,0}-\sum_{j=1}^{d-2}
\mbox{proj}_{\ell^{j}}\ell^{d-1,0},\quad\ell^{d-1}=
{\tilde\ell^{d-1}}/{|\tilde\ell^{d-1}|},
\end{split}
\end{equation}
the vector field is orthogonal to each other. Now we define the corresponding unit normal direction which is orthogonal to $\ell^{k,0}$, $k=1,\ldots,d-1,$ (and thus also $\ell^k$):
\begin{equation}\label{defnorm}
{\bf n}(x)=(n^1,\ldots,n^d)^{\top}=\frac{(-\ell_d^{1,0},\ldots,-\ell_d^{d-1,0},1)^{\top}}{\big(1+\sum_{k=1}^{d-1}(\ell_d^{k,0})^2\big)^{1/2}}.
\end{equation}
Obviously, ${\bf n}(x)={\bf n}_j$ on $\Gamma_j$.

For any point $x_0\in B_{3/4}\cap \cD_{j_{0}}$, $j_0=1,\ldots,m+1$, suppose  the closest point on $\partial \cD_{j_{0}}$ to $x_0$ is $y_0:=(y'_0,h_{j_{0}}(y'_0))$.  On the surface $\Gamma_{j_0}$, the unit normal vector at $(y'_0,h_{j_0}(y'_0))$ is
\begin{equation}\label{defny0}
{\bf n}_{y_0}=(n^1_{y_0},\ldots, n_{y_0}^d)^{\top}=
\frac{\big(-\nabla_{x'}h_{j_0}(y'_0),1\big)^{\top}}{\big(1+|\nabla_{x'}h_{j_0}(y'_0)|^{2})^{1/2}}.
\end{equation}
The corresponding tangential vectors are defined by
\begin{equation}\label{deftauk}
\tau_k=\ell^{k}(y_0),\quad k=1,\ldots,d-1,
\end{equation}
where $\ell^{k}$ is defined in \eqref{defell}. In the coordinate system associated with $x_0$ with the axes parallelled to ${\bf n}_{y_0}$ and $\tau_k,k=1,\ldots,d-1$, we will use $y=(y',y^d)$ and $D_{y}$ to denote the point and the derivatives, respectively. Moreover, we have $y=\Lambda x$, where $$\Lambda=(\Lambda^1,\ldots,\Lambda^d)^\top=(\Lambda^{\alpha\beta})_{\alpha,\beta=1}^{d}$$
is a $d\times d$ matrix representing the linear transformation from the coordinate system associated with $0$ to the coordinate system associated with $x_0$, and $\tau_k=(\Gamma^{1k},\dots,\Gamma^{dk})^\top,k=1,\ldots,d-1$, ${\bf n}_{y_0}=(\Gamma^{1d},\dots,\Gamma^{dd})^\top$, where $\Gamma=\Lambda^{-1}$.
Finally, we introduce  $m+1$ ``strips'' (in the $y$-coordinates)
$$\Omega_j:=\{y\in\cD: y_{j-1}^d<y^{d}<y_j^d\},\quad j=1,\ldots,m+1,$$
where $y_j:=(\Lambda'y_0,y_j^d)\in \Gamma_j$ and  $\Lambda'=(\Lambda^1,\ldots,\Lambda^{d-1})^\top$. For any $0<r\leq1/4$, we have
\begin{equation}\label{volume}
|(\cD_{j}\setminus\Omega_{j})\cap (B_{r}(\Lambda x_0))|\leq Nr^{d+1/2},\quad j=1,\ldots,m+1.
\end{equation}
See, for instance, \cite[Lemma 2.3]{dx2019}.

\subsection{Auxiliary results}

Here we collect some elementary results. The following weak type-$(1,1)$ estimate is almost the same as \cite[Lemma 3.4]{cd2019}.

\begin{lemma}\label{weak est barv}
Let $q\in(1,\infty)$. Let $({\bf v},\pi)\in W_0^{1,q}(B_{r}(\Lambda x_0))^{d}\times L_0^q(B_{r}(\Lambda x_0))$ be a weak solution to
\begin{align*}
\begin{cases}
D_{\alpha}(\overline{{\mathcal A}^{\alpha\beta}}(y^{d})D_{\beta}{\bf v})+D\pi={\boldsymbol{\mathfrak f}}\mathbbm{1}_{B_{r/2}(\Lambda x_0)}+D_\alpha({\bf F}^\alpha\mathbbm{1}_{B_{r/2}(\Lambda x_0)})\\
\Div{\bf v}=\mathcal H\mathbbm{1}_{B_{r/2}(\Lambda x_0)}-(\mathcal H\mathbbm{1}_{B_{r/2}(\Lambda x_0)})_{B_{r}(\Lambda x_0)}
\end{cases}\,\, \mbox{in }B_{r}(\Lambda x_0),
\end{align*}
where ${\boldsymbol{\mathfrak f}}, {\bf F}^\alpha, \mathcal H\in L^{q}(B_{r/2}(\Lambda x_0))$. Then for any $t>0$, we have
\begin{align*}
|\{y\in B_{r/2}(\Lambda x_0): |D{\bf v}(y)|+|\pi(y)|>t\}|\leq\frac{N}{t}\int_{B_{r/2}(\Lambda x_0)}\left(|{\bf F}^\alpha|+|\mathcal H|+r|{\boldsymbol{\mathfrak f}}|\right)\, dy,
\end{align*}
where $N=N(d,q,\nu)$.
\end{lemma}

\begin{lemma}\cite[Theorem 2.4]{cdx2022}\label{lemlocbdd}
Let $\varepsilon\in (0,1)$, $q\in(1,\infty)$,  $A^{\alpha\beta}$, ${\bf f}^{\alpha}$, and $g$ satisfy Assumption \ref{assump} with $s=0$. Let $({\bf u},p)\in W^{1,q}(B_1)^d\times L^q(B_1)$ be a weak solution to \eqref{stokes} in $B_1$. Then $({\bf u},p)\in C^{1,\delta'}(B_{1-\varepsilon}\cap\overline{{\cD}_{j_0}})^d\times C^{\delta'}(B_{1-\varepsilon}\cap\overline{{\cD}_{j_0}})$ and it holds that
\begin{align*}
&\|D{\bf u}\|_{L^{\infty}(B_{1/4})}+|{\bf u}|_{1,\delta'; B_{1-\varepsilon}\cap\overline{{\cD}_{j_0}}}+\|p\|_{L^{\infty}(B_{1/4})}+|p|_{\delta'; B_{1-\varepsilon}\cap\overline{{\cD}_{j_0}}}\\
&\leq N\big(\|D{\bf u}\|_{L^{1}(B_1)}+\|p\|_{L^{1}(B_1)}+\sum_{j=1}^{M}|{\bf f}^\alpha|_{\delta;\overline{\cD_{j}}}+\sum_{j=1}^{M}|g|_{\delta;\overline{\cD_{j}}}\big),
\end{align*}
where $j_0=1,\ldots,m+1$, $\delta'=\min\{\delta,\frac{\mu}{1+\mu}\}$, $N>0$ is a constant depending only on $d,m,q,\nu,\varepsilon$, $|A|_{\delta;\overline{\cD_{j}}}$, and the $C^{1,\mu}$ norm of $h_j$.
\end{lemma}

\section{A new Stokes system}\label{newsystem}
This section is devoted to deriving a new Stokes system in $B_{3/4}$ as follows:
\begin{align}\label{eqtildeu}
\begin{cases}
D_\alpha(A^{\alpha\beta}D_\beta \tilde {\bf u})+D\tilde p={\bf f}+ D_\alpha \tilde {\bf f}^\alpha,\\
\Div\tilde {\bf u}=D_\ell g+D\ell_i D_i{\bf u}-\sum_{j=1}^{m+1}\mathbbm{1}_{_{\cD_j}}D\ell_{i,j}D_i{\bf u}(P_jx_0)-\sum_{j=1}^{m+1}(\mathbbm{1}_{_{\cD_j^c}}D \tilde\ell_{i,j}D_i{\bf u}(P_jx_0))_{B_1},
\end{cases}
\end{align}
where $\tilde {\bf u}$ and $\tilde p$ are defined in \eqref{def-tildeu}, ${\bf f}$ and $\tilde {\bf f}^\alpha$ are defined in \eqref{defg} and \eqref{def-tildefalpha}, respectively, and $\tilde\ell_{,j}:=(\tilde\ell_{1,j},\dots,\tilde\ell_{d,j})$ is a smooth extension of $\ell|_{\cD_j}$ to $\cup_{k=1,k\neq j}^{m+1}\cD_k$.

To prove \eqref{eqtildeu}, we first use the definition of weak solutions to find that the problem \eqref{stokes} is equivalent to a homogeneous transmission problem
\begin{align}
                        \label{eq2.57}
\begin{cases}
D_\alpha(A^{\alpha\beta}D_\beta {\bf u})+Dp=D_\alpha {\bf f}^\alpha \qquad \text{in}\,\,\bigcup_{j=1}^{m+1}\cD_j, \\
{\bf u}|_{\Gamma_j}^+={\bf u}|_{\Gamma_j}^-,\quad[n^\alpha_j(A^{\alpha\beta} D_\beta {\bf u} -{\bf f}^\alpha)+p{\bf n}_j]_{\Gamma_j}=0,\quad j=1,\ldots,m,\\
\Div {\bf u}=g\qquad \text{in}\,\,\bigcup_{j=1}^{m+1}\cD_j,
\end{cases}
\end{align}
where
\begin{align*}
&[n_j^\alpha(A^{\alpha\beta} D_\beta {\bf u} -{\bf f}^\alpha)+p{\bf n}_j]_{\Gamma_j}\\
&:=(n_j^\alpha(A^{\alpha\beta} D_\beta {\bf u} -{\bf f}^\alpha)+p{\bf n}_j)|_{\Gamma_j}^+-(n_j^\alpha(A^{\alpha\beta} D_\beta {\bf u} -{\bf f}^\alpha)+p{\bf n}_j)|_{\Gamma_j}^-,
\end{align*}
${\bf n}_j$ is the unit normal vector on $\Gamma_j$ defined by \eqref{normal}, ${\bf u}|_{\Gamma_j}^+$ and ${\bf u}|_{\Gamma_j}^-$ ($n_j^\alpha A^{\alpha\beta} D_\beta {\bf u} |_{\Gamma_j}^+$ and $n_j^\alpha A^{\alpha\beta} D_\beta {\bf u} |_{\Gamma_j}^-$) are the left and right limits of ${\bf u}$ (its conormal derivatives) on $\Gamma_j$, respectively, $j=1,\ldots,m$. Here and throughout this paper the superscript
$\pm$ indicates the limit from outside and inside the domain, respectively. Taking the directional derivative of \eqref{eq2.57}
along the direction $\ell:=\ell^k$,  $k=1,\ldots,d-1$, we get the following inhomogeneous transmission problem
\begin{align}\label{eqsecond}
\begin{cases}
D_\alpha(A^{\alpha\beta}D_\beta D_\ell {\bf u})+DD_\ell p={\bf f}+ D_\alpha {\bf f}^{\alpha,1} \quad\text{in}\,\,\bigcup_{j=1}^{m+1}\cD_j, \\
D_\ell {\bf u}|_{\Gamma_j}^+=D_\ell {\bf u}|_{\Gamma_j}^-,\quad[n_j^\alpha (A^{\alpha\beta} D_\beta D_\ell {\bf u}-{\bf f}^{\alpha,1})+{\bf n}_jD_\ell p]_{\Gamma_j}=\tilde {\bf h}_j,~j=1,\dots,m,\\
\Div(D_\ell {\bf u})=D_\ell g+D\ell_i D_i{\bf u}\qquad \text{in}\,\,\bigcup_{j=1}^{m+1}\cD_j,
\end{cases}
\end{align}
where
\begin{equation}\label{defg}
\begin{split}
{\bf f}&=(A^{\alpha\beta} D_\beta D{\bf u}+DA^{\alpha\beta}D_\beta {\bf u}-D{\bf f}^\alpha)D_\alpha\ell+D\ell Dp,\\
{\bf f}^{\alpha,1}&=D_\ell {\bf f}^\alpha+A^{\alpha\beta}(D_\beta \ell_i) D_i{\bf u}-D_{\ell} A^{\alpha\beta}D_\beta {\bf u},
\end{split}
\end{equation}
and
\begin{align}\label{deftildeh}
\tilde {\bf h}_j=[D_\ell n_j^\alpha (-A^{\alpha\beta}D_\beta {\bf u}+{\bf f}^\alpha)-pD_\ell {\bf n}_j]_{\Gamma_j}.
\end{align}
From \eqref{normal}, it follows that $D_\ell {\bf n}_j$ is a tangential direction on $\Gamma_j$ and thus we may write $\tilde {\bf h}_j=\tilde {\bf h}_j(x')$ and  $D_{\ell}{\bf n}_j\in C^{\mu}$.

Now by  adding a term
$$
\sum_{j=1}^{m}D_d\Big(\mathbbm{1}_{x^d>h_j(x')} {\tilde {\bf h}_j(x')}/{n^d_j(x')}\Big)
$$
to the first equation in \eqref{eqsecond}, where $\mathbbm{1}_{\bullet}$ is  the indicator function, we can get rid of $\tilde {\bf h}_j$ in the second equation of \eqref{eqsecond} and reduce the problem \eqref{eqsecond} to a homogeneous transmission problem:
\begin{align}\label{homosecond0}
\begin{cases}
D_\alpha(A^{\alpha\beta}D_\beta D_\ell {\bf u})+DD_\ell p={\bf f}+ D_\alpha {\bf f}^{\alpha,2} \quad\text{in}\,\,\bigcup_{j=1}^{m+1}\cD_j, \\
D_\ell {\bf u}|_{\Gamma_j}^+=D_\ell {\bf u}|_{\Gamma_j}^-,\quad[n_j^\alpha (A^{\alpha\beta} D_\beta D_\ell {\bf u}-{\bf f}^{\alpha,2})+{\bf n}_jD_\ell p]_{\Gamma_j}=0,\\
\Div(D_\ell {\bf u})=D_\ell g+D\ell_i D_i{\bf u}\qquad \text{in}\,\,\bigcup_{j=1}^{m+1}\cD_j,
\end{cases}
\end{align}
where
\begin{align*}
{\bf f}^{\alpha,2}:={\bf f}^{\alpha,1}+\delta_{\alpha d}\sum_{j=1}^{m}\mathbbm{1}_{x^d>h_j(x')} \frac{\tilde {\bf h}_j(x')}{n^d_j(x')},
\end{align*}
$\delta_{\alpha d}=1$ if $\alpha=d$, and $\delta_{\alpha d}=0$ if $\alpha\neq d$. Note that $D\ell$ is singular at any point where two interfaces touch or are very close to each other. To cancel out this singularity, for $x_{0}\in B_{3/4}\cap \overline{\cD_{j_{0}}}$, we consider
\begin{equation}\label{tildeu}
{\bf u}_{_\ell}:={\bf u}_{_\ell}(x;x_0)=D_{\ell}{\bf u}-{\bf u}_0,
\end{equation}
where
\begin{align}\label{defu0}
{\bf u}_0:={\bf u}_0(x;x_0)=\sum_{j=1}^{m+1}\tilde\ell_{i,j}D_i{\bf u}(P_jx_0),
\end{align}
\begin{align}\label{Pjx}
P_jx_0=\begin{cases}
x_0&\quad\mbox{for}\quad j=j_0,\\
(x'_0,h_j(x'_0))&\quad\mbox{for}\quad j<j_0,\\
(x'_0,h_{j-1}(x'_0))&\quad\mbox{for}\quad j>j_0,
\end{cases}
\end{align}
and the vector field
$\tilde\ell_{,j}:=(\tilde\ell_{1,j},\dots,\tilde\ell_{d,j})$ is a smooth extension of $\ell|_{\cD_j}$ to $\cup_{k=1,k\neq j}^{m+1}\cD_k$.
Then it follows from \eqref{homosecond0} that
\begin{align}\label{homosecond}
\begin{cases}
D_\alpha(A^{\alpha\beta}D_\beta {\bf u}_{_\ell})+DD_\ell p={\bf f}+ D_\alpha {\bf f}^{\alpha,3} &\quad\text{in}\,\,\bigcup_{j=1}^{m+1}\cD_j, \\
[n_j^\alpha (A^{\alpha\beta} D_\beta {\bf u}_{_\ell}-{\bf f}^{\alpha,3})+{\bf n}_jD_\ell p]_{\Gamma_j}=0,~\quad j=1,\dots,m,\\
\Div {\bf u}_{_\ell}=D_\ell g+D\ell_i D_i{\bf u}-\sum_{j=1}^{m+1}D \tilde\ell_{i,j}D_i{\bf u}(P_jx_0)&\quad \text{in}\,\,\bigcup_{j=1}^{m+1}\cD_j,
\end{cases}
\end{align}
where
\begin{align}\label{tildef1}
{\bf f}^{\alpha,3}&:={\bf f}^{\alpha,3}(x;x_0)={\bf f}^{\alpha,2}-A^{\alpha\beta}\sum_{j=1}^{m+1}D_\beta \tilde\ell_{i,j}D_i{\bf u}(P_jx_0)\nonumber\\
&=D_\ell {\bf f}^\alpha-D_{\ell} A^{\alpha\beta}D_\beta {\bf u}+A^{\alpha\beta}\big(D_\beta \ell_i D_i {\bf u}-\sum_{j=1}^{m+1}D_\beta \tilde\ell_{i,j}D_i{\bf u}(P_jx_0)\big)\nonumber\\
&\quad+\delta_{\alpha d}\sum_{j=1}^{m}\mathbbm{1}_{x^d>h_j(x')} (n^d_j(x'))^{-1}\tilde {\bf h}_j(x').
\end{align}

Note that the mean oscillation of
$$A^{\alpha\beta}\big(D_\beta \ell_i D_i {\bf u}-\sum_{j=1}^{m+1}D_\beta \tilde\ell_{i,j}D_i{\bf u}(P_jx_0)\big)$$
in \eqref{tildef1} is only bounded. For this, we choose a cut-off function $\zeta\in C_{0}^\infty(B_1)$ satisfying
$$0\leq\zeta\leq1,\quad\zeta\equiv 1~\mbox{in}~B_{3/4},\quad|D\zeta|\leq8.$$
Denote
\begin{equation}\label{mathcalA}
\tilde A^{\alpha\beta}:=\zeta A^{\alpha\beta}+\nu(1-\zeta)\delta_{\alpha\beta}\delta_{ij}.
\end{equation}
For $j=1,\ldots,m+1$, denote $\cD_j^c:=\cD\setminus\cD_j$. From \cite[Corollary 5.3]{cl2017}, it follows that there exists $({\boldsymbol{\mathfrak u}}_j(\cdot;x_0),{\mathfrak \pi}_j(\cdot;x_0))\in W^{1,q}(B_1)^d\times L_0^q(B_1)$ such that
\begin{align}\label{eq-rmu}
\begin{cases}
D_{\alpha}(\tilde A^{\alpha\beta}D_\beta {\boldsymbol{\mathfrak u}}_j(\cdot;x_0))+D{\mathfrak \pi}_j(\cdot;x_0)=-D_{\alpha}(\mathbbm{1}_{_{\cD_j^c}}A^{\alpha\beta}D_\beta \tilde\ell_{i,j}D_i{\bf u}(P_jx_0))&\,\, \mbox{in}~B_1,\\
\Div{\boldsymbol{\mathfrak u}}_j(\cdot;x_0)=-\mathbbm{1}_{_{\cD_j^c}}D \tilde\ell_{i,j}D_i{\bf u}(P_jx_0)+(\mathbbm{1}_{_{\cD_j^c}}D \tilde\ell_{i,j}D_i{\bf u}(P_jx_0))_{B_1}&\,\, \mbox{in}~B_1,\\
{\boldsymbol{\mathfrak u}}_j(\cdot;x_0)=0&\,\, \mbox{on}~\partial B_1,
\end{cases}
\end{align}
where $1<q<\infty$. Moreover, by using the fact that  $\mathbbm{1}_{_{\cD_j^c}}D_\beta \tilde\ell_{,j}$ is piecewise $C^{\mu}$ and the local boundedness estimate of $D{\bf u}$  in Lemma \ref{lemlocbdd}, it holds that
\begin{align}\label{rmuj}
&\|{\boldsymbol{\mathfrak u}}_j(\cdot;x_0)\|_{W^{1,q}(B_1)}+\|{\mathfrak \pi}_j(\cdot;x_0)\|_{L^{q}(B_1)}\nonumber\\
&\leq N\|\mathbbm{1}_{_{\cD_j^c}}A^{\alpha\beta}D_\beta \tilde\ell_{i,j}D_i{\bf u}(P_jx_0)\|_{L^q(B_1)}+N\|\mathbbm{1}_{_{\cD_j^c}}D \tilde\ell_{i,j}D_i{\bf u}(P_jx_0)\|_{L^q(B_1)}\nonumber\\
&\leq N\big(\|D{\bf u}\|_{L^{1}(B_1)}+\|p\|_{L^{1}(B_1)}+\sum_{j=1}^{M}|{\bf f}^\alpha|_{1,\delta;\overline{\cD_{j}}}+\sum_{j=1}^{M}|g|_{1,\delta; \overline{\cD_{j}}}\big),
\end{align}
where $N>0$ is a constant depending on $d,m,q,\nu,\varepsilon$, $|A|_{\delta;\overline{\cD_{j}}}$, and the $C^{1,\mu}$ norm of $h_j$. We also obtain from Lemma \ref{lemlocbdd} that
$$
({\boldsymbol{\mathfrak u}}_j(\cdot;x_0),{\mathfrak \pi}_j(\cdot;x_0))\in C^{1,\mu'}(\overline{\cD_i}\cap B_{1-\varepsilon})^d\times C^{\mu'}(\overline{\cD_i}\cap B_{1-\varepsilon}),\quad i=1,\dots,m+1,
$$
with the estimate
\begin{align*}
&\|D{\boldsymbol{\mathfrak u}}_j\|_{L^\infty(B_{1/4})}+|{\boldsymbol{\mathfrak u}}_j|_{1,\mu';\overline{\cD_i}\cap B_{1-\varepsilon}}+\|{\mathfrak \pi}_j\|_{L^\infty(B_{1/4})}+|{\mathfrak \pi}_j|_{\mu';\overline{\cD_i}\cap B_{1-\varepsilon}}\nonumber\\
&\leq N\big(\|D{\boldsymbol{\mathfrak u}}_j(\cdot;x_0))\|_{L^{1}(B_1)}+\|{\mathfrak \pi}_j(\cdot;x_0))\|_{L^{1}(B_1)}+|\mathbbm{1}_{_{\cD_j^c}}A^{\alpha\beta}D_\beta \tilde\ell_{i,j}D_i{\bf u}(t_0,P_jx_0)|_{\mu;\overline{\cD_{j}}}\nonumber\\
&\quad+|\mathbbm{1}_{_{\cD_j^c}}D \tilde\ell_{i,j}D_i{\bf u}(t_0,P_jx_0)|_{\mu;\overline{\cD_{j}}}\big)\nonumber\\
&\le N\big(\|D{\bf u}\|_{L^{1}(B_1)}+\|p\|_{L^{1}(B_1)}+\sum_{j=1}^{M}|{\bf f}^\alpha|_{1,\delta;\overline{\cD_{j}}}+\sum_{j=1}^{M}|g|_{1,\delta; \overline{\cD_{j}}}\big),
\end{align*}
where $\mu':=\min\{\mu,\frac{1}{2}\}$ and we used \eqref{rmuj} in the second inequality.

Denote
\begin{equation*}
{\boldsymbol{\mathfrak u}}:={\boldsymbol{\mathfrak u}}(x;x_0)=\sum_{j=1}^{m+1}{\boldsymbol{\mathfrak u}}_j(x;x_0),\quad {\mathfrak \pi}:={\mathfrak \pi}(x;x_0)=\sum_{j=1}^{m+1}{\mathfrak \pi}_j(x;x_0).
\end{equation*}
Then for each $i=1,\ldots,m+1$, we have
\begin{align}\label{estauxiu}
&\|D{\boldsymbol{\mathfrak u}}\|_{L^\infty(B_{1/4})}+|{\boldsymbol{\mathfrak u}}|_{1,\mu';\overline{\cD_i}\cap B_{1-\varepsilon}}+\|{\mathfrak \pi}\|_{L^\infty(B_{1/4})}+|{\mathfrak \pi}|_{\mu';\overline{\cD_i}\cap B_{1-\varepsilon}}\nonumber\\
&\le N\big(\|D{\bf u}\|_{L^{1}(B_1)}+\|p\|_{L^{1}(B_1)}+\sum_{j=1}^{M}|{\bf f}^\alpha|_{1,\delta;\overline{\cD_{j}}}+\sum_{j=1}^{M}|g|_{1,\delta; \overline{\cD_{j}}}\big).
\end{align}
We further define
\begin{equation}\label{def-tildeu}
\tilde {\bf u}:=\tilde {\bf u}(x;x_0)={\bf u}_{_\ell}-{\boldsymbol{\mathfrak u}}=D_{\ell}{\bf u}-{\bf u}_0-{\boldsymbol{\mathfrak u}},\quad \tilde p:=\tilde p(x;x_0)=D_\ell p-{\mathfrak \pi}.
\end{equation}
Then $(\tilde {\bf u},\tilde p)$ satisfies \eqref{eqtildeu}, where
\begin{align}\label{def-tildefalpha}
\tilde {\bf f}^\alpha:=\tilde {\bf f}^\alpha(x;x_0)=\tilde {\bf f}^{\alpha,1}(x;x_0)+\tilde {\bf f}^{\alpha,2}(x),
\end{align}
with
\begin{align}\label{deff1}
\tilde {\bf f}^{\alpha,1}(x;x_0):=A^{\alpha\beta}\big(D_\beta \ell_i D_i {\bf u}-\sum_{j=1}^{m+1}\mathbbm{1}_{_{\cD_j}}D_\beta \ell_{i}D_i{\bf u}(P_jx_0)\big),
\end{align}
and
\begin{align}\label{deff2}
\tilde {\bf f}^{\alpha,2}(x):=D_\ell {\bf f}^\alpha-D_{\ell} A^{\alpha\beta}D_\beta {\bf u}+\delta_{\alpha d}\sum_{j=1}^{m}\mathbbm{1}_{x^d>h_j(x')} (n^d_j(x'))^{-1}\tilde {\bf h}_j(x').
\end{align}
Compared to \eqref{tildef1}, such data $\tilde {\bf f}^\alpha$ is good enough for us to apply Campanato's method in  \cite{c1963,g1983}, since the mean oscillation of $\tilde {\bf f}^\alpha$ vanishes at a certain rate as the radii of the balls go to zero (see the proof of \eqref{estF} below for the details).

\section{Decay estimates}\label{auxilemma}
Let us denote
\begin{equation}\label{deftildeU}
\tilde {\bf U}:=\tilde {\bf U}(x;x_0)=n^\alpha(A^{\alpha\beta}D_\beta \tilde {\bf u}-\tilde {\bf f}^\alpha)+{\bf n}\tilde p,
\end{equation}
where
$n^\alpha$ and ${\bf n}$ are defined in \eqref{defnorm}, $\alpha=1,\ldots,d$. Denote
\begin{equation}\label{defPhi}
\Phi(x_0,r):=\inf_{\mathbf q^{k'},\mathbf Q\in\mathbb R^{d}}\left(\fint_{B_r(x_0)}\big(|D_{\ell^{k'}}\tilde {\bf u}(x;x_0)-\mathbf q^{k'}|^{\frac{1}{2}}+|\tilde {\bf U}(x;x_0)-\mathbf Q|^{\frac{1}{2}}\big)\,dx \right)^{2},
\end{equation}
where $\tilde {\bf u}$ and $\tilde {\bf U}$ are defined in \eqref{def-tildeu} and \eqref{deftildeU}, respectively.
We shall  adapt the argument in \cite{dx2022} to establish a decay estimate of
\begin{align}\label{def-phi}
\phi(\Lambda x_0,r):=\inf_{\mathbf q^{k'},\mathbf Q\in\mathbb R^{d}}\Big(\fint_{B_r(\Lambda x_0)}\big(|D_{y^{k'}}\tilde{\bf v}(y;\Lambda x_0)-\mathbf q^{k'}|^{\frac{1}{2}}+|\tilde{\bf V}(y;\Lambda x_0)-\mathbf Q|^{\frac{1}{2}}\big)\,dy\Big)^{2},
\end{align}
where
\begin{equation}\label{tildeV}
\tilde{\bf V}(y;\Lambda x_0)=\mathcal{A}^{d\beta}D_{y^\beta}\tilde{\bf v}(y;\Lambda x_0)-\tilde {\boldsymbol{\mathfrak f}}^d(y;\Lambda x_0)+\tilde{\mathfrak p}(y;\Lambda x_0){\bf e}_d,
\end{equation}
${\bf e}_d$ is the $d$-th unit vector in $\mathbb R^d$, $\tilde {\boldsymbol{\mathfrak f}}^\alpha=(\tilde{\mathfrak f}_1^\alpha,\dots,\tilde{\mathfrak f}_d^\alpha)^{\top}$ with $\alpha=1,\dots,d$,
\begin{equation}\label{transformation}
\begin{split}
&\mathcal{A}^{\alpha\beta}(y)=\Lambda\Lambda^{\alpha k}A^{ks}(x)\Lambda^{s\beta}\Gamma,\quad \tilde {\bf v}(y;\Lambda x_0)=\Lambda\tilde {\bf u}(x;x_0),\quad \tilde{\mathfrak p}(y;\Lambda x_0)=\tilde p(x;x_0), \\
&\tilde{\mathfrak f}_\tau^\alpha(y;\Lambda x_0)=\Lambda^{\tau m}\Lambda^{\alpha k}\tilde f_m^k(x;x_0),\quad\tau=1,\dots,d,
\end{split}
\end{equation}
$\tilde f_m^k(x;x_0)$ is the $m$-th component of $\tilde{\bf f}^k(x;x_0)$ defined in \eqref{def-tildefalpha} with $k$ in place of $\alpha$, $y=\Lambda x$, $\Lambda=(\Lambda^{\alpha\beta})_{\alpha,\beta=1}^{d}$ is defined in Section \ref{secpreliminaries} (see p.\pageref{deftauk}), and $\Gamma=\Lambda^{-1}$. Denote
\begin{equation}\label{defG}
G:=G(x;x_0)=D_\ell g+D\ell_i D_i{\bf u}-\sum_{j=1}^{m+1}\mathbbm{1}_{_{\cD_j}}D\ell_{i,j}D_i{\bf u}(P_jx_0)-\sum_{j=1}^{m+1}(\mathbbm{1}_{_{\cD_j^c}}D \tilde\ell_{i,j}D_i{\bf u}(P_jx_0))_{B_1},
\end{equation}
and set
\begin{equation}\label{def-G}
\mathcal G:=\mathcal G(y;\Lambda x_0)=G(x;x_0),\quad {\boldsymbol{\mathfrak f}}=({\mathfrak f}_1,\dots,{\mathfrak f}_d)^{\top},~{\mathfrak f}_\tau(y)=\Lambda^{\tau m} f_m(x).
\end{equation}
Then it follows  from \eqref{eqtildeu} that $\tilde{\bf v}$ satisfies
\begin{equation}\label{eqfraku}
\begin{cases}
D_\alpha(\mathcal{A}^{\alpha\beta}D_\beta \tilde{\bf v})+D\tilde{\mathfrak p}={\boldsymbol{\mathfrak f}}+ D_\alpha \tilde {\boldsymbol{\mathfrak f}}^\alpha\\
\Div\tilde {\bf v}=\mathcal G
\end{cases}\,\,\mbox{in}~\Lambda(B_{3/4}),
\end{equation}
where $\tilde {\boldsymbol{\mathfrak f}}^\alpha=(\tilde{\mathfrak f}_1^\alpha,\dots,\tilde{\mathfrak f}_d^\alpha)^{\top}$. From \eqref{transformation}, the $\tau$-th component of $\tilde {\boldsymbol{\mathfrak f}}^{\alpha,1}$ and $\tilde {\boldsymbol{\mathfrak f}}^{\alpha,2}$ is
\begin{equation}\label{deff1f2}
\tilde {\mathfrak f}_\tau^{\alpha,1}(y;\Lambda x_0)=\Lambda^{\tau m}\Lambda^{\alpha k}\tilde f_m^{k,1}(x;x_0),\quad \tilde {\mathfrak f}_\tau^{\alpha,2}(y)=\Lambda^{\tau m}\Lambda^{\alpha k}\tilde f_m^{k,2}(x),
\end{equation}
where $\tilde f_m^{k,1}(x;x_0)$ and $\tilde f_m^{k,2}(x)$ are the $m$-th component of $\tilde {\bf f}^{k,1}(x;x_0)$ and $\tilde {\bf f}^{k,2}(x)$ defined in \eqref{deff1} and \eqref{deff2}, respectively. Then $\tilde {\boldsymbol{\mathfrak f}}^{\alpha,1}+\tilde {\boldsymbol{\mathfrak f}}^{\alpha,2}=\tilde {\boldsymbol{\mathfrak f}}^\alpha$ which is defined in \eqref{transformation}.

Recalling that ${\bf f}^\alpha, A^{\alpha\beta}\in C^{1,\delta}(\overline{\cD_{j}})$, $D_{\ell}n_j^\alpha\in C^{\mu}$, the assumption that $D{\bf u}$ and $p$ are piecewise $C^1$, and the fact that the vector field $\ell$ is $C^{1/2}$ (see \cite[Lemma 2.1]{dx2022}), we find that  $\tilde {\boldsymbol{\mathfrak f}}^{\alpha,2}$ is piecewise $C^{\delta_{\mu}}$, where $\delta_{\mu}=\min\big\{\frac{1}{2},\mu,\delta\big\}$. Now we denote
\begin{align}\label{def-Falpha}
&{\bf F}^\alpha:={\bf F}^\alpha(y;\Lambda x_0)=(\overline{{\mathcal A}^{\alpha\beta}}(y^{d})-\mathcal{A}^{\alpha\beta}(y))D_{y^\beta}\tilde{\bf v}(y;\Lambda x_0)+\tilde {\boldsymbol{\mathfrak f}}^{\alpha,1}(y;\Lambda x_0)+\tilde {\boldsymbol{\mathfrak f}}^{\alpha,2}(y)-\bar{\boldsymbol{\mathfrak f}}^{\alpha,2}(y^d),\nonumber\\
& {\bf F}=(F^1,\ldots,F^d),\quad\mathcal H:=\mathcal G-\overline{\mathcal G},
\end{align}
where $\bar{\boldsymbol{\mathfrak f}}^{\alpha,2}(y^d)$ and $\overline{\mathcal G}$ are piecewise constant functions corresponding to $\tilde {\boldsymbol{\mathfrak f}}^{\alpha,2}(y)$ and $\mathcal G$, respectively. For the convenience of notation,
set
\begin{align}\label{defC1}
\mathcal{C}_1:=\sum_{j=1}^{m+1}\|D^2{\bf u}\|_{L^\infty(B_{r}(x_0)\cap\cD_j)}+\sum_{j=1}^{m+1}|{\bf f}^\alpha|_{1,\delta;\overline{\cD_{j}}}+\sum_{j=1}^{m+1}|g|_{1,\delta; \overline{\cD_{j}}}+\|D{\bf u}\|_{L^{1}(B_1)}+\|p\|_{L^{1}(B_1)},
\end{align}
and
\begin{align}\label{defC0}
\mathcal{C}_0:=\mathcal{C}_1+\sum_{j=1}^{m+1}\|Dp\|_{L^\infty(B_{r}(x_0)\cap\cD_j)}.
\end{align}

\begin{lemma}\label{lemmaFG}
Let ${\boldsymbol{\mathfrak f}}$, ${\bf F}$, and $\mathcal H$ be defined as in  \eqref{def-G} and \eqref{def-Falpha}, respectively. Then we have
\begin{align}\label{estfL1}
\|\boldsymbol{\mathfrak f}\|_{L^{1}(B_{r}(\Lambda x_0))}
\leq N\mathcal{C}_0 r^{d-\frac{1}{2}},
\end{align}
\begin{align}\label{estF}
\|{\bf F}\|_{L^1(B_{r}(\Lambda x_0))}\leq N\mathcal{C}_0r^{d+\delta_{\mu}},
\end{align}
and
\begin{align}\label{LemmaH}
\|\mathcal H\|_{L^1(B_{r}(\Lambda x_0))}\leq N\mathcal{C}_1r^{d+\delta_{\mu}},
\end{align}
where  $\mathcal{C}_0$ and $\mathcal{C}_1$ are defined in \eqref{defC0} and \eqref{defC1}, respectively,  $\delta_{\mu}=\min\big\{\frac{1}{2},\mu,\delta\big\}$, $N$ depends on  $|A|_{1,\delta;\overline{\cD_{j}}}$,  $d,q,m,\nu$, and the $C^{2,\mu}$ norm of $h_j$.
\end{lemma}

\begin{proof}
Note that
\begin{align}\label{estDellk0}
\int_{B_r(x_0)\cap\cD_j}|D\ell|\,dx\leq Nr^{d-\frac{1}{2}},
\end{align}
see \cite[(3.26)]{dx2022}. Here, $N$ depends only on the $C^{2,\mu}$ norm of $h_j$. Then together with ${\mathfrak f}_\tau(y)=\Lambda^{\tau m}f_m(x)$, Lemma \ref{lemlocbdd}, and  \eqref{defg}, we obtain \eqref{estfL1}.
	
Since $\tilde {\boldsymbol{\mathfrak f}}^{\alpha,2}(y)$ is piecewise $ C^{\delta_{\mu}}$, we  have
\begin{align}\label{estDf}
\int_{B_{r}(\Lambda x_0)}\big|\tilde {\boldsymbol{\mathfrak f}}^{\alpha,2}(y)-\bar{\boldsymbol{\mathfrak f}}^{\alpha,2}(y^d)\big|
&\leq Nr^{d+\delta_{\mu}}\Big(\sum_{j=1}^{m+1}\|D^2{\bf u}\|_{L^\infty(B_{r}(x_0)\cap\cD_j)}+\sum_{j=1}^{m+1}\|Dp\|_{L^\infty(B_{r}(x_0)\cap\cD_j)}\nonumber\\
&\quad+\sum_{j=1}^{m+1}|{\bf f}^\alpha|_{1,\delta;\overline{\cD_{j}}}+\|D{\bf u}\|_{L^{1}(B_1)}+\|p\|_{L^{1}(B_1)}\Big),
\end{align}
where $\delta_{\mu}=\min\big\{\frac{1}{2},\mu,\delta\big\}$, and  $N$ depends on $d,m$, and the $C^{2,\mu}$ norm of $h_j$.
By using \eqref{defu0} and \eqref{def-tildeu}, we have
\begin{align*}
D_s\tilde {\bf u}(x;x_0)=\ell_iD_{s}D_i{\bf u}-D_{s}{\boldsymbol{\mathfrak u}}+D_s \ell_i D_i{\bf u}-\sum_{j=1}^{m+1}D_s \tilde\ell_{i,j}D_i{\bf u}(P_jx_0).
\end{align*}
Then combining with \eqref{transformation}, we have
\begin{align*}
&(\overline{{\mathcal A}^{\alpha\beta}}(y^{d})-\mathcal{A}^{\alpha\beta}(y))D_{y^\beta}\tilde{\bf{ v}}(y;\Lambda x_0)\\
&=(\overline{{\mathcal A}^{\alpha\beta}}(y^{d})-\mathcal{A}^{\alpha\beta}(y))\Lambda\Gamma^{\beta s}D_s\tilde {\bf u}(x;x_0)\\
&=(\overline{{\mathcal A}^{\alpha\beta}}(y^{d})-\mathcal{A}^{\alpha\beta}(y))\Lambda\Gamma^{\beta s}\big(\ell_iD_{s}D_i{\bf u}-D_{s}{\boldsymbol{\mathfrak u}}+D_s \ell_i D_i{\bf u}-\sum_{j=1}^{m+1}D_s \tilde\ell_{i,j}D_i{\bf u}(P_jx_0)\big).
\end{align*}
Using \eqref{deff1}, \eqref{deff1f2}, and $\mathcal{A}^{\alpha\beta}(y)=\Lambda\Lambda^{\alpha k}A^{ks}(x)\Lambda^{s\beta}\Gamma$ in \eqref{transformation}, we have for each $\tau=1,\dots,d$,
\begin{align*}
\tilde {\mathfrak f}_\tau^{\alpha,1}(y;\Lambda x_0)=\Lambda^{\tau m}\Lambda^{\alpha k}\tilde f_m^{k,1}(x;x_0)
&=\Lambda^{\tau m}\Lambda^{\alpha k}A_{mn}^{ks}(x)\big(D_s \ell_i D_i {\bf u}^n-\sum_{j=1}^{m+1}\mathbbm{1}_{_{\cD_j}}D_s \ell_{i}D_i{\bf u}^n(P_jx_0)\big)\nonumber\\
&=\mathcal{A}_{\tau\gamma}^{\alpha\beta}(y)\Lambda^{\gamma n}\Gamma^{\beta s}\big(D_s \ell_i D_i {\bf u}^n-\sum_{j=1}^{m+1}
\mathbbm{1}_{_{\cD_j}}D_s \ell_{i}D_i{\bf u}^n(P_jx_0)\big).
\end{align*}
Thus,
$$\tilde {\boldsymbol{\mathfrak f}}^{\alpha,1}(y;\Lambda x_0)=\mathcal{A}^{\alpha\beta}(y)\Lambda\Gamma^{\beta s}\big(D_s \ell_i D_i {\bf u}-\sum_{j=1}^{m+1}
\mathbbm{1}_{_{\cD_j}}D_s \ell_{i}D_i{\bf u}(P_jx_0)\big)$$
and
\begin{align*}
&(\overline{{\mathcal A}^{\alpha\beta}}(y^{d})-\mathcal{A}^{\alpha\beta}(y))D_{y^\beta}\tilde{\bf{ v}}(y;\Lambda x_0)+\tilde {\boldsymbol{\mathfrak f}}^{\alpha,1}(y;\Lambda x_0)\nonumber\\
&=(\overline{{\mathcal A}^{\alpha\beta}}(y^{d})-\mathcal{A}^{\alpha\beta}(y))\Lambda\Gamma^{\beta s}(\ell_iD_{s}D_i{\bf u}-D_{s}{\boldsymbol{\mathfrak u}}-\sum_{j=1}^{m+1}\mathbbm{1}_{_{\cD_j^c}}D_s \tilde\ell_{i,j}D_i{\bf u}(P_jx_0))\nonumber\\
&\quad+\overline{{\mathcal A}^{\alpha\beta}}(y^{d})\Lambda\Gamma^{\beta s}
\big(D_s \ell_i D_i{\bf u}-\sum_{j=1}^{m+1}
\mathbbm{1}_{_{\cD_j}}
D_s\ell_{i}D_i{\bf u}(P_jx_0)\big).
\end{align*}
Together with $\mathcal{A}\in C^{1,\delta}(\cD_{\varepsilon}\cap\overline{\cD}_j)$,  \eqref{volume}, \eqref{estauxiu}, \eqref{estDellk0}, and the fact that $\mathbbm{1}_{_{\cD_j^c}}D_s \tilde\ell_{i,j}$ is piecewise $C^\mu$, we have
\begin{align*}
\|(\overline{{\mathcal A}^{\alpha\beta}}(y^{d})-\mathcal{A}^{\alpha\beta}(y))D_{y^\beta}\tilde{\bf{ v}}(y;\Lambda x_0)+\tilde {\boldsymbol{\mathfrak f}}^{\alpha,1}(y;\Lambda x_0)\|_{L^1(B_{r}(\Lambda x_0))}\leq N\mathcal{C}_1r^{d+\frac{1}{2}}.
\end{align*}
Combining with \eqref{estDf}, we derive \eqref{estF}.

Finally, recalling $\mathcal G:=\mathcal G(y;\Lambda x_0)=G(x;x_0)$, where $G(x;x_0)$ is defined in \eqref{defG}, using \eqref{volume}, Lemma \ref{lemlocbdd}, and the fact that $\mathbbm{1}_{_{\cD_j^c}}D\tilde\ell_{i,j}$ is piecewise $C^\mu$ again, we have \eqref{LemmaH}.
The proof of the lemma is complete.
\end{proof}

\begin{lemma}\label{lemma iteraphi}
Let $\varepsilon\in(0,1)$ and $q\in(1,\infty)$. Suppose that $A^{\alpha\beta}$, ${\bf f}^\alpha$, and $g$ satisfy Assumption \ref{assump} with $s=1$. If $(\tilde{\bf v},\tilde{\mathfrak p})$ is a weak solution to \eqref{eqfraku},
then for any $0<\rho\leq r\leq 1/4$, we have
\begin{align*}
\phi(\Lambda x_0,\rho)&\leq N\Big(\frac{\rho}{r}\Big)^{\delta_{\mu}}\phi(\Lambda x_0,r/2)+N\mathcal{C}_0\rho^{\delta_{\mu}},
\end{align*}
where $\phi(\Lambda x_0,r)$ is defined in \eqref{def-phi}, $\mathcal{C}_0$ is defined in \eqref{defC0}, $\delta_{\mu}=\min\big\{\frac{1}{2},\mu,\delta\big\}$, $N$ depends on $d,m,q,\nu$, the $C^{2,\mu}$ norm of $h_j$, and $|A|_{1,\delta;\overline{\cD_{j}}}$.
\end{lemma}

\begin{proof}
Let ${\bf v_0}=(v_0^1,\dots,v_0^d)$ and $p_0$ be functions of $y^d$, such that $v_0^d=\overline{\mathcal G}$, $\overline{{\mathcal A}^{dd}}{\bf v}_0+p_0{\bf e}_d=\bar{\boldsymbol{\mathfrak f}}_2^d$, where $\overline{\mathcal G}$ and $\bar{\boldsymbol{\mathfrak f}}_2^d$ are piecewise constant functions corresponding to $\mathcal G$ and $\tilde{\boldsymbol{\mathfrak f}}_2^d$, respectively. Set
$${\bf v}_e=\tilde{\bf v}-\int_{\Lambda x_0^d}^{y^d}{\bf v}_0(s)\,ds,\quad p_e=\tilde {\mathfrak p}-p_0.$$
Then according with  \eqref{eqfraku}, we have
\begin{equation*}
\begin{cases}
D_\alpha(\overline{{\mathcal A}^{\alpha\beta}}(y^{d})D_\beta {\bf v}_e)+Dp_e={\boldsymbol{\mathfrak f}}+ D_\alpha {\bf F}^\alpha\\
\Div{\bf v}_e=\mathcal H
\end{cases}\,\,\mbox{in}~B_{r}(\Lambda x_0),
\end{equation*}
where $\mathcal H=\mathcal G-\overline{\mathcal G}$, and ${\bf F}^\alpha$ is defined in \eqref{def-Falpha}.
Now we decompose $({\bf v}_e,p_e)=({\bf v},p_1)+({\bf w},p_2)$, where $({\bf v},p_1)\in W_0^{1,q}(B_{r}(\Lambda x_0))^d\times L_0^q(B_{r}(\Lambda x_0))$ satisfies
\begin{align*}
\begin{cases}
D_{\alpha}(\overline{{\mathcal A}^{\alpha\beta}}(y^d)D_{\beta}{\bf v})+Dp_1={\boldsymbol{\mathfrak f}}\mathbbm{1}_{B_{r/2}(\Lambda x_0)}+D_\alpha({\bf F}^\alpha\mathbbm{1}_{B_{r/2}(\Lambda x_0)})\\
\Div{\bf v}=\mathcal H\mathbbm{1}_{B_{r/2}(\Lambda x_0)}-(\mathcal H\mathbbm{1}_{B_{r/2}(\Lambda x_0)})_{B_{r}(\Lambda x_0)}
\end{cases}\,\, \mbox{in}~B_{r}(\Lambda x_0).
\end{align*}
Then by  Lemmas \ref{weak est barv} and \ref{lemmaFG}, we have
\begin{align}\label{holder v}
\left(\fint_{B_{r/2}(\Lambda x_0)}(|D{\bf v}|+|p_1|)^{\frac{1}{2}}\,dy\right)^2\leq N\mathcal{C}_0r^{\delta_{\mu}},
\end{align}
where  $\mathcal{C}_0$ is defined in \eqref{defC0}. Moreover, $({\bf w},p_2)$ satisfies
\begin{align*}
\begin{cases}
D_{\alpha}(\overline{{\mathcal A}^{\alpha\beta}}(y^d)D_{\beta}{\bf w})+Dp_2=0\\
\Div{\bf w}=(\mathcal H\mathbbm{1}_{B_{r/2}(\Lambda x_0)})_{B_{r}(\Lambda x_0)}
\end{cases}\,\, \mbox{in}~B_{r/2}(\Lambda x_0).
\end{align*}
Then it follows from \cite[(3.7)]{cd2019} that
\begin{align}\label{diff-wW}
&\Bigg(\fint_{B_{\kappa r}(\Lambda x_0)}\big(|D_{y^{k'}}{\bf w}(y;\Lambda x_0)-(D_{y^{k'}}{\bf w})_{B_{\kappa r}(\Lambda x_{0})}|^{\frac{1}{2}}+|{\bf W}(y;\Lambda x_0)-({\bf W})_{B_{\kappa r}(\Lambda x_{0})}|^{\frac{1}{2}}\big)\,dy\Bigg)^{2}\nonumber\\
&\leq N\kappa\left(\fint_{B_{r/2}(\Lambda x_0)}\big(|D_{y^{k'}}{\bf w}(y;\Lambda x_0)-\mathbf{q}^{k'}|^{\frac{1}{2}}+|{\bf W}(y;\Lambda x_0)-\mathbf Q|^{\frac{1}{2}}\big)\,dy\right)^{2},
\end{align}
where ${\bf W}:={\bf W}(y;\Lambda x_0)=\overline{\mathcal A^{d\beta}}(y^d)D_{y^\beta}{\bf w}(y;\Lambda x_0)+p_2{\bf e}_d$ and $\kappa\in(0,1/2)$ to be fixed later. Set
\begin{equation*}
{\bf V}_e=\overline{\mathcal A^{d\beta}}(y^d)D_{y^\beta}{\bf v}_e(y;\Lambda x_0)+p_e{\bf e}_d.
\end{equation*}
Then
\begin{align*}
\tilde{\bf V}-{\bf V}_e=-{\bf F}^d(y;\Lambda x_0),
\end{align*}
where $\tilde{\bf V}$ and ${\bf F}^d$ are defined in \eqref{tildeV} and \eqref{def-Falpha}, respectively. Thus, combining  the triangle inequality, \eqref{holder v}, \eqref{diff-wW}, and \eqref{estF}, we obtain
\begin{align*}
&\Bigg(\fint_{B_{\kappa r}(\Lambda x_0)}\big(|D_{y^{k'}}\tilde{\bf v}(y;\Lambda x_0)-(D_{y^{k'}}{\bf w})_{B_{\kappa r}(\Lambda x_{0})}|^{\frac{1}{2}}+|\tilde{\bf V}(y;\Lambda x_0)-({\bf W})_{B_{\kappa r}(\Lambda x_{0})}|^{\frac{1}{2}}\big)\,dy\Bigg)^{2}\nonumber\\
&\leq N\kappa\left(\fint_{B_{r/2}(\Lambda x_0)}\big(|D_{y^{k'}}\tilde{\bf v}(y;\Lambda x_0)-\mathbf{q}^{k'}|^{\frac{1}{2}}+|\tilde{\bf V}(y;\Lambda x_0)-\mathbf Q|^{\frac{1}{2}}\big)\,dy\right)^{2}\nonumber\\
&\quad+N\kappa^{-2d}\left(\fint_{B_{r/2}(\Lambda x_0)}|{\bf F}^d(y;\Lambda x_0)|^{\frac{1}{2}}\,dy\right)^{2}+N\kappa^{-2d}\mathcal{C}_0r^{\delta_{\mu}}\\
&\leq N\kappa\left(\fint_{B_{r/2}(\Lambda x_0)}\big(|D_{y^{k'}}\tilde{\bf v}(y;\Lambda x_0)-\mathbf{q}^{k'}|^{\frac{1}{2}}+|\tilde{\bf V}(y;\Lambda x_0)-\mathbf Q|^{\frac{1}{2}}\big)\,dy\right)^{2}+N\kappa^{-2d}\mathcal{C}_0r^{\delta_{\mu}}.
\end{align*}
Using the fact that $\mathbf{q}^{k'}, \mathbf{Q}\in\mathbb R^{d}$ are arbitrary, we deduce
\begin{align*}
\phi(\Lambda x_0,\kappa r)\leq N_{0}\kappa\phi(\Lambda x_0,r/2)+N\kappa^{-2d}\mathcal{C}_0r^{\delta_{\mu}}.
\end{align*}
Choosing $\kappa\in(0,1/2)$  small enough so that $N_{0}\kappa\leq\kappa^{\gamma}$ for any fixed $\gamma\in(\delta_{\mu},1)$ and iterating, we get
\begin{align*}
\phi(\Lambda x_0,\kappa^{j}r)\leq\kappa^{j\delta_{\mu}}\phi(\Lambda x_0,r/2)+N\mathcal{C}_0(\kappa^{j}r)^{\delta_{\mu}}.
\end{align*}
Therefore, for any $\rho$ with $0<\rho\leq r\leq1/4$ and $\kappa^j r\leq\rho<\kappa^{j-1}r$, we have
\begin{align*}
\phi(\Lambda x_0,\rho)\leq N\Big(\frac{\rho}{r}\Big)^{\delta_{\mu}}\phi(\Lambda x_0,r/2)+N\mathcal{C}_0\rho^{\delta_{\mu}}.
\end{align*}
The lemma is proved.
\end{proof}

Now we are ready to prove the decay estimate of $\Phi(x_{0},r)$ defined  in \eqref{defPhi} as follows.

\begin{lemma}\label{lemma itera}
Let $\varepsilon\in(0,1)$ and $q\in(1,\infty)$. Suppose that $A^{\alpha\beta}$, ${\bf f}^\alpha$, and $g$ satisfy Assumption \ref{assump} with $s=1$. If $(\tilde{\bf u},\tilde p)$ is a weak solution to \eqref{eqtildeu},
then for any $0<\rho\leq r\leq 1/4$, we have
\begin{align}\label{est phi'}
\Phi(x_{0},\rho)&\leq N\Big(\frac{\rho}{r}\Big)^{\delta_{\mu}}\Phi(x_{0},r/2)+N\mathcal{C}_0\rho^{\delta_{\mu}},
\end{align}
where $\mathcal{C}_0$ is defined in \eqref{defC0}, $\delta_{\mu}=\min\big\{\frac{1}{2},\mu,\delta\big\}$, $N$ depends on $d,m,q,\nu$, the $C^{2,\mu}$ norm of $h_j$, and $|A|_{1,\delta;\overline{\cD_{j}}}$.
\end{lemma}

\begin{proof}
The proof is an adaptation of \cite[Lemma 3.4]{dx2022}. Let $y_0$ be as in Section \ref{secpreliminaries}. Note that
\begin{equation}\label{Dell-nalpha}
\begin{split}
&D_{\ell^k}\tilde {\bf u}(x;x_0)-\Gamma D_{y^k}\tilde{\bf v}(y;\Lambda x_0)=(\ell^k(x)-\tau_k)\cdot D\tilde {\bf u}(x;x_0),\\
&\tilde {\bf U}(x;x_0)-\Gamma\big(\mathcal{A}^{d\beta}(y)D_{y^\beta}\tilde{\bf v}(y;\Lambda x_0)-\tilde{\boldsymbol{\mathfrak f}}^d(y;\Lambda x_0)+\tilde{\mathfrak p}(y;\Lambda x_0){\bf e}_d\big)\\
&=(n^\alpha-n_{y_0}^\alpha)(A^{\alpha\beta}(x)D_\beta \tilde {\bf u}(x;x_0)-\tilde {\bf f}^\alpha(x;x_0))+({\bf n}-{\bf n}_{y_0})\tilde p(x;x_0),
\end{split}
\end{equation}
where  $\tau_k$ and $n_{y_0}^\alpha$ are defined in \eqref{deftauk} and \eqref{defny0}, respectively. For any $x\in B_r(x_0)\cap\cD_{j}$, where $r\in(|x_0-y_0|,1)$ and $j=1,\ldots,m+1$, we have
\begin{align*}
|\ell^k(x)-\tau_{k}|\leq N\sqrt r,\quad |{\bf n}(x)-{\bf n}_{y_0}|\leq N\sqrt r,
\end{align*}
where $k=1,\ldots,d-1$. See the proof of \cite[Lemma 3.4]{dx2022} for the details.
Then coming back to \eqref{Dell-nalpha}, we obtain
\begin{equation}\label{difference-coor}
\begin{aligned}
&|D_{\ell^k}\tilde {\bf u}(x;x_0)-\Gamma D_{y^k}\tilde{\bf v}(y;\Lambda x_0)|\leq N\sqrt r|D\tilde {\bf u}(x;x_0)|,\\
&|\tilde {\bf U}(x;x_0)-\Gamma(\mathcal{A}^{d\beta}(y)D_{y^\beta}\tilde{\bf v}(y;\Lambda x_0)-\tilde{\boldsymbol{\mathfrak f}}^d(y;\Lambda x_0)+\tilde{\mathfrak p}(y;\Lambda x_0){\bf e}_d)|\\
&\leq N\sqrt r(|D\tilde {\bf u}(x;x_0)|+|\tilde {\bf f}^\alpha(x;x_0)|+|\tilde p(x;x_0)|).
\end{aligned}
\end{equation}
By using  \eqref{def-tildeu}, \eqref{defu0}, and \eqref{estauxiu}, we have
\begin{align}\label{meanDtildeu}
\fint_{B_{r}(x_{0})}(|D\tilde {\bf u}|+|\tilde p|)\,dx&\leq \sum_{j=1}^{m+1}\|D^2{\bf u}\|_{L^\infty(B_{r}(x_0)\cap\cD_j)}+\sum_{j=1}^{m+1}\|Dp\|_{L^\infty(B_{r}(x_0)\cap\cD_j)}+\|D{\boldsymbol{\mathfrak u}}\|_{L^\infty(B_{r}(x_0))}\nonumber\\
&\quad+\|\pi\|_{L^\infty(B_{r}(x_0))}+\fint_{B_{r}(x_{0})}\big|D\ell^k D{\bf u}-\sum_{j=1}^{m+1}D\tilde\ell_{,j}^kD{\bf u}(P_jx_0)\big|\,dx\nonumber\\
&\leq \sum_{j=1}^{m+1}\|D^2{\bf u}\|_{L^\infty(B_{r}(x_0)\cap\cD_j)}+\sum_{j=1}^{m+1}\|Dp\|_{L^\infty(B_{r}(x_0)\cap\cD_j)}\nonumber\\
&\quad+N\big(\|D{\bf u}\|_{L^{1}(B_1)}+\|p\|_{L^{1}(B_1)}+\sum_{j=1}^{M}|{\bf f}^\alpha|_{1,\delta;\overline{\cD_{j}}}+\sum_{j=1}^{M}|g|_{1,\delta; \overline{\cD_{j}}}\big)\nonumber\\
&\quad+\fint_{B_{r}(x_{0})}\big|D\ell^k D{\bf u}-\sum_{j=1}^{m+1}D\tilde\ell_{,j}^kD{\bf u}(P_jx_0)\big|\,dx.
\end{align}	
To estimate the last term on the right-hand side above, on one hand, using the fact that $\tilde\ell_{,j}$ is the smooth extension of $\ell|_{\cD_j}$ to $\cup_{k=1,k\neq j}^{m+1}\cD_k$ and the local boundedness of $D{\bf u}$ in Lemma \ref{lemlocbdd}, we obtain
\begin{align}\label{estDellk}
&\Big\|\sum_{j=1,j\neq i}^{m+1}D\tilde\ell_{,j}^kD{\bf u}(P_jx_0)\Big\|_{L_1(B_{r}(x_0)\cap\cD_i)}\nonumber\\
&\leq Nr^{d}\big(\|D{\bf u}\|_{L^{1}(B_1)}+\|p\|_{L^{1}(B_1)}+\sum_{j=1}^{M}|{\bf f}^\alpha|_{1,\delta;\overline{\cD_{j}}}
+\sum_{j=1}^{M}|g|_{1,\delta; \overline{\cD_{j}}}\big),
\end{align}
where $i=1,\ldots,m+1$. On the other hand, it follows from \eqref{estDellk0} that
\begin{align}\label{estDellk00}
&\big\|D\ell^k (D{\bf u}-D{\bf u}(P_ix_0))\big\|_{L_1(B_{r}(x_0)\cap\cD_i)}\nonumber\\
&\leq Nr\|D^2{\bf u}\|_{L^\infty(B_{r}(x_0)\cap\cD_i)}\int_{B_r(x_0)\cap\cD_i}|D\ell^k|\,dx\leq Nr^{d+\frac{1}{2}}\|D^2{\bf u}\|_{L^\infty(B_{r}(x_0)\cap\cD_i)}.
\end{align}
Thus, coming back to \eqref{meanDtildeu}, and using \eqref{estDellk} and \eqref{estDellk00}, we obtain
\begin{align}\label{estfintDtildeu}
\fint_{B_{r}(x_{0})}(|D\tilde {\bf u}|+|\tilde p|)\,dx\leq N\mathcal{C}_0,
\end{align}
where $\mathcal{C}_0$ is defined in \eqref{defC0}. It follows from  \eqref{def-tildefalpha} that
\begin{align*}
\fint_{B_{r}(x_{0})}|\tilde {\bf f}^\alpha|\,dx
&\leq N\big(\sum_{j=1}^{M}|{\bf f}^\alpha|_{1,\delta;\overline{\cD_{j}}}+\|D{\bf u}\|_{L^{1}(B_1)}\big)\nonumber\\
&\quad+N\fint_{B_{r}(x_0)}\big|D_\beta \ell_i D_i{\bf u}-\sum_{j=1}^{m+1}\mathbbm{1}_{_{\cD_j}}D_\beta \ell_{i}D_i{\bf u}(P_jx_0)\big|\,dx\nonumber\\
&\quad+\fint_{B_{r}(x_{0})}\big|\delta_{\alpha d}\sum_{j=1}^{m}\mathbbm{1}_{x^d>h_j(x')} (n^d_j(x'))^{-1}\tilde {\bf h}_j(x')\big|\,dx.
\end{align*}
Then by using \eqref{estDellk0} and \eqref{deftildeh}, we derive
\begin{align}\label{estfinttildef}
\fint_{B_{r}(x_{0})}|\tilde {\bf f}^\alpha|\,dx\leq N\mathcal{C}_1,
\end{align}
where $\mathcal{C}_1$ is defined in \eqref{defC1}.

Using the triangle inequality and \eqref{difference-coor}--\eqref{estfinttildef}, we have
\begin{align*}
&\left(\fint_{B_\rho(x_0)}\big(|D_{\ell^{k'}}\tilde {\bf u}(x;x_0)-\mathbf q^{k'}|^{\frac{1}{2}}+|\tilde {\bf U}(x;x_0)-\mathbf Q|^{\frac{1}{2}}\big)\,dx \right)^{2}\\
&\leq \Bigg(\fint_{B_\rho(\Lambda x_0)}\big(|\Gamma(D_{y^{k'}}\tilde{\bf v}(y;\Lambda x_0)-\Gamma^{-1}\mathbf q^{k'})|^{\frac{1}{2}}\nonumber\\
&\qquad+|\Gamma(\mathcal{A}^{d\beta}(y)D_{y^\beta}\tilde{\bf v}(y;\Lambda x_0)-\tilde{\boldsymbol{\mathfrak f}}^d(y;\Lambda x_0)+\tilde{\mathfrak p}(y;\Lambda x_0){\bf e}_d-\Gamma^{-1}\mathbf Q)|^{\frac{1}{2}}\Big)\,dy\Bigg)^{2}+N\mathcal{C}_0\sqrt{\rho}\\
&\leq \Bigg(\fint_{B_\rho(\Lambda x_0)}\big(|D_{y^{k'}}\tilde{\bf v}(y;\Lambda x_0)-\Gamma^{-1}\mathbf q^{k'}|^{\frac{1}{2}}\nonumber\\
&\qquad+|\mathcal{A}^{d\beta}(y)D_{y^\beta}\tilde{\bf v}(y;\Lambda x_0)-\tilde{\boldsymbol{\mathfrak f}}^d(y;\Lambda x_0)+\tilde{\mathfrak p}(y;\Lambda x_0){\bf e}_d-\Gamma^{-1}\mathbf Q|^{\frac{1}{2}}\Big)\,dy\Bigg)^{2}+N\mathcal{C}_0\sqrt{\rho},
\end{align*}
where $0<\rho\leq r\leq 1/4$ and $\mathcal{C}_0$ is defined in \eqref{defC0}.
By using the fact that $\mathbf{q}^{k'}, \mathbf{Q}\in\mathbb R^{d}$ are arbitrary, we obtain
\begin{align*}
\Phi(x_{0},\rho)\leq \phi(\Lambda x_{0},\rho)+N\mathcal{C}_0\sqrt{\rho}.
\end{align*}
Combining with  Lemma \ref{lemma iteraphi}, we derive
\begin{align}\label{est phiPhi}
\Phi(x_{0},\rho)\leq N\Big(\frac{\rho}{r}\Big)^{\delta_{\mu}}\phi(\Lambda x_{0},r/2)+N\mathcal{C}_0\rho^{\delta_{\mu}}.
\end{align}
Similarly, we have
\begin{align*}
\phi(\Lambda x_{0},r/2)\leq\Phi(x_{0},r/2)+N\mathcal{C}_0\sqrt{r}.
\end{align*}
Substituting it into \eqref{est phiPhi} and using $\delta_{\mu}\leq1/2$, we obtain
\begin{align*}
\Phi(x_{0},\rho)\leq N\Big(\frac{\rho}{r}\Big)^{\delta_{\mu}}\Phi(x_{0},r/2)+N\mathcal{C}_0\rho^{\delta_{\mu}}.
\end{align*}	
The lemma is proved.
\end{proof}

\section{The boundedness of \texorpdfstring{$\|D^2{\bf u}\|_{L^\infty}$}{} and \texorpdfstring{$\|Dp\|_{L^\infty}$}{} }\label{bddestimate}

For convenience, set
\begin{align}\label{defC3}
\mathcal{C}_2:=\|D{\bf u}\|_{L^{1}(B_1)}
+\|p\|_{L^{1}(B_1)}+\sum_{j=1}^{m+1}|{\bf f}^\alpha|_{1,\delta;\overline{\cD_{j}}}+\sum_{j=1}^{m+1}|g|_{1,\delta;\overline{\cD_{j}}}.
\end{align}
We first prove the estimates of $\|D\tilde{\bf u}(\cdot;x_0)\|_{L^2(B_{r/2}(x_0))}$ and $\|\tilde p(\cdot;x_0)\|_{L^2(B_{r/2}(x_0))}$ in the following lemma.

\begin{lemma}\label{lemmaup}
Under the same assumptions as in Lemma \ref{lemma itera}, we have
\begin{align*}
&\|D\tilde{\bf u}(\cdot;x_0)\|_{L^2(B_{r/2}(x_0))}+\|\tilde p(\cdot;x_0)\|_{L^2(B_{r/2}(x_0))}\nonumber\\
&\leq Nr^{\frac{d+1}{2}}\Big(\sum_{j=1}^{m+1}\|D^2{\bf u}\|_{L^\infty(B_{r}(x_0)\cap\cD_j)}+\sum_{j=1}^{m+1}\|Dp\|_{L^\infty(B_{r}(x_0)\cap\cD_j)}\Big)+N\mathcal{C}_2r^{\frac{d}{2}-1},
\end{align*}
where $x_0\in \cD_{\varepsilon}\cap{{\cD}_{j_0}}$, $r\in(0,1/4)$,  $\tilde{\bf u}$ and $\tilde p$ are defined in \eqref{def-tildeu}, the constant $N>0$ depends on $d,m,q,\nu,\varepsilon$, $|A|_{1,\delta;\overline{\cD_{j}}}$, and the $C^{2,\mu}$ norm of $h_j$.
\end{lemma}

\begin{proof}
We start with proving the estimate of $\|D\tilde{\bf u}(\cdot;x_0)\|_{L^2(B_{r/2}(x_0))}$. By using the definition of weak solutions, the transmission problem \eqref{homosecond} is equivalent to
\begin{align}\label{equell}
\begin{cases}
D_\alpha(A^{\alpha\beta}D_\beta {\bf u}_{_\ell})+D(D_\ell p-(D_\ell p)_{B_{r}(x_0)})={\bf f}+ D_\alpha {\bf f}^{\alpha,3}\\
\Div {\bf u}_{_\ell}=D_\ell g+D\ell_i D_i{\bf u}-\sum_{j=1}^{m+1}D \tilde\ell_{i,j}D_i{\bf u}(P_jx_0)
\end{cases}\,\, \text{in}\,B_1.
\end{align}
By \cite[Lemma 10]{art2005}, one can find $\boldsymbol\psi\in H_0^1(B_{r}(x_0))^d$ satisfying
$$\Div\boldsymbol\psi=D_\ell p-(D_\ell p)_{B_{r}(x_0)}\quad\mbox{in}~B_{r}(x_0),$$
and
\begin{equation}\label{estellp}
\|\boldsymbol\psi\|_{L^2(B_{r}(x_0))}+r\|D\boldsymbol\psi\|_{L^2(B_{r}(x_0))}\leq Nr\|D_\ell p-(D_\ell p)_{B_{r}(x_0)}\|_{L^2(B_{r}(x_0))},
\end{equation}
where  $N=N(d)$. Then by applying $\boldsymbol\psi$ to \eqref{equell} as a test function, and using Young's inequality and \eqref{estellp}, we obtain
\begin{align*}
&\int_{B_{r}(x_0)}|D_\ell p-(D_\ell p)_{B_{r}(x_0)}|^2\,dx\\
&=-\int_{B_{r}(x_0)}A^{\alpha\beta}D_\beta {\bf u}_{_\ell}D_\alpha\boldsymbol{\psi}\,dx-\int_{B_{r}(x_0)}{\bf f}\boldsymbol\psi\,dx+\int_{B_{r}(x_0)}{\bf f}^{\alpha,3} D_\alpha\boldsymbol\psi\,dx\\
&\leq\varepsilon_0\int_{B_{r}(x_0)}|D_\ell p-(D_\ell p)_{B_{r}(x_0)}|^2\, dx+N(d,\varepsilon_0)\int_{B_{r}(x_0)}(|D{\bf u}_{_\ell}|^2+r^2|{\bf f}|^2+|{\bf f}^{\alpha,3}|^2)\,dx.
\end{align*}
Taking $\varepsilon_0=\frac{1}{2}$, we have
\begin{align}\label{estDellp}
\|D_\ell p-(D_\ell p)_{B_{r}(x_0)}\|_{L^2(B_{r}(x_0))}\leq N\big(\|D{\bf u}_{_\ell}\|_{L^2(B_{r}(x_0))}+r\|{\bf f}\|_{L^2(B_{r}(x_0))}+\|{\bf f}^{\alpha,3}|\|_{L^2(B_{r}(x_0))}\big).
\end{align}

Now we choose $\eta\in C_0^\infty(B_r(x_0))$ such that
\begin{equation}\label{defeta}
0\leq\eta\leq1,\quad\eta=1\quad\mbox{in}~B_{r/2}(x_0),\quad |D\eta|\leq \frac{N(d)}{r}.
\end{equation}
Then we apply $\eta^2{\bf u}_{_\ell}$ to \eqref{equell} as a test function to obtain
\begin{align*}
&\int_{B_{r}(x_0)}\eta^2A^{\alpha\beta}D_\beta {\bf u}_{_\ell}D_\alpha{\bf u}_{_\ell}\,dx\\
&=-2\int_{B_{r}(x_0)}\eta {\bf u}_{_\ell}A^{\alpha\beta}D_\beta {\bf u}_{_\ell}D_\alpha\eta\, dx-2\int_{B_{r}(x_0)}\eta{\bf u}_{_\ell}D\eta(D_\ell p-(D_\ell p)_{B_{r}(x_0)})\,dx\\
&\quad-\int_{B_{r}(x_0)}\eta^2(D_\ell p-(D_\ell p)_{B_{r}(x_0)})(D_\ell g+D\ell_i D_i{\bf u}-\sum_{j=1}^{m+1}D \tilde\ell_{i,j}D_i{\bf u}(P_jx_0))\,dx\\
&\quad-\int_{B_{r}(x_0)}{\bf f}\eta^2{\bf u}_{_\ell}\,dx+\int_{B_{r}(x_0)}\eta^2{\bf f}^{\alpha,3} D_\alpha{\bf u}_{_\ell}\,dx+2\int_{B_{r}(x_0)}\eta{\bf f}^{\alpha,3} D_\alpha\eta{\bf u}_{_\ell}\,dx.
\end{align*}
Using the ellipticity condition, Young's inequality, \eqref{estDellp}, and $\eta=1$ in $B_{r/2}(x_0)$, we derive
\begin{align*}
\|D{\bf u}_{_\ell}(\cdot;x_0)\|_{L^2(B_{r/2}(x_0))}
&\leq N\big(r^{-1}\|{\bf u}_{_\ell}(\cdot;x_0)\|_{L^2(B_{r}(x_0))}+r\|{\bf f}\|_{L^2(B_{r}(x_0))}+\|{\bf f}^{\alpha,3}(\cdot;x_0)\|_{L^2(B_{r}(x_0))}\nonumber\\
&\quad+\|D_\ell g+D\ell_i D_i{\bf u}-\sum_{j=1}^{m+1}D \tilde\ell_{i,j}D_i{\bf u}(P_jx_0)\|_{L^2(B_{r}(x_0))}\big)\\
&\quad+\varepsilon_1\|D{\bf u}_{_\ell}(\cdot;x_0)\|_{L^2(B_{r}(x_0))},
\end{align*}
where $\varepsilon_1>0$. This, in combination with a well-known iteration argument (see, for instance, \cite[pp. 81--82]{g93}), yields
\begin{align}\label{est-Dul}
\|D{\bf u}_{_\ell}(\cdot;x_0)\|_{L^2(B_{r/2}(x_0))}
&\leq N\big(r^{-1}\|{\bf u}_{_\ell}(\cdot;x_0)\|_{L^2(B_{r}(x_0))}+r\|{\bf f}\|_{L^2(B_{r}(x_0))}+\|{\bf f}^{\alpha,3}(\cdot;x_0)\|_{L^2(B_{r}(x_0))}\nonumber\\
&\quad+\|D_\ell g+D\ell_i D_i{\bf u}-\sum_{j=1}^{m+1}D \tilde\ell_{i,j}D_i{\bf u}(P_jx_0)\|_{L^2(B_{r}(x_0))}\big).
\end{align}

Next we estimate the terms on the right-hand side above. By using \eqref{tildeu} and the local boundedness estimate of $D{\bf u}$  in Lemma \ref{lemlocbdd}, we obtain
\begin{equation}\label{est-uell}
\|{\bf u}_{_\ell}(\cdot;x_0)\|_{L^2(B_{r}(x_0))}\leq Nr^{\frac{d}{2}}\Big(\|D{\bf u}\|_{L^{1}(B_1)}+\|p\|_{L^{1}(B_1)}+\sum_{j=1}^{M}|{\bf f}^\alpha|_{1,\delta;\overline{\cD_{j}}}+\sum_{j=1}^{M}|g|_{1,\delta;\overline{\cD_{j}}}\Big).
\end{equation}
From the definition of $\ell$ in \eqref{defell}, it follows that
\begin{align}\label{estDellk02}
\int_{B_r(x_0)\cap\cD_i}|D\ell^k|^2\,dx\leq N\int_{B_r'(x'_0)}\frac{\min\{2r,h_i-h_{i-1}\}}{|h_{i}-h_{i-1}|}\, dx'\leq Nr^{d-1}.
\end{align}
See \cite[lemma 2.1]{dx2022} for the properties of $\ell$. Using this and recalling
the definition of ${\bf f}$ in \eqref{defg},  we get
\begin{align}\label{est-g1}
\|{\bf f}\|_{L^2(B_{r}(x_0))}\leq N\mathcal{C}_0r^{\frac{d-1}{2}},
\end{align}
where $\mathcal{C}_0$ is defined in \eqref{defC0}. Similarly, we have
\begin{align*}
\left(\fint_{B_{r}(x_{0})}\big|D\ell_i D_i{\bf u}-\sum_{j=1}^{m+1}D \tilde\ell_{i,j}D_i{\bf u}(P_jx_0)\big|^2\,dx\right)^{1/2}\leq Nr^{1/2}\sum_{j=1}^{m+1}\|D^2{\bf u}\|_{L^\infty(B_{r}(x_0)\cap\cD_j)}+N\mathcal{C}_2,
\end{align*}
where $\mathcal{C}_2$ is defined in \eqref{defC3}. According to \eqref{tildef1}, we have
\begin{align}\label{est-f1breve}
\|{\bf f}^{\alpha,3}(\cdot;x_0)\|_{L^2(B_{r}(x_0))}\leq Nr^{\frac{d+1}{2}}\sum_{j=1}^{m+1}\|D^2{\bf u}\|_{L^\infty(B_{r}(x_0)\cap\cD_j)}+N\mathcal{C}_2r^{\frac{d}{2}}.
\end{align}
Thus, substituting \eqref{est-uell}, \eqref{est-g1}, and \eqref{est-f1breve} into \eqref{est-Dul}, we obtain
\begin{align}\label{estDuell}
&\|D{\bf u}_{_\ell}(\cdot;x_0)\|_{L^2(B_{r/2}(x_0))}\notag\\
&\leq Nr^{\frac{d+1}{2}}\Big(\sum_{j=1}^{m+1}\|D^2{\bf u}\|_{L^\infty(B_{r}(x_0)\cap\cD_j)}+\sum_{j=1}^{m+1}\|Dp\|_{L^\infty(B_{r}(x_0)\cap\cD_j)}\Big)
+N\mathcal{C}_2r^{\frac{d}{2}-1}.
\end{align}
Combining \eqref{estDuell} with \eqref{estauxiu} and \eqref{def-tildeu}, the estimate of $\|D\tilde{\bf u}(\cdot;x_0)\|_{L^2(B_{r/2}(x_0))}$ follows.

Next we proceed to estimate $\|\tilde p\|_{L^2(B_{r/2}(x_0))}$.  Integrating $D_\ell (p\eta^2)$ over $B_r(x_0)$ directly, and using the integration by parts and $\eta\in C_0^\infty(B_r(x_0))$, we obtain
\begin{equation*}
\int_{B_{r}(x_0)}\big(D_\ell (p\eta^2)+p\eta^2\Div\ell\big) \,dx=0.
\end{equation*}
Then by using \cite[Lemma 10]{art2005} again, there exists a function $\boldsymbol\varphi\in H_0^1(B_{r}(x_0))^d$ such that
$$\Div\boldsymbol\varphi=D_\ell (p\eta^2)+p\eta^2\Div\ell\quad\mbox{in}~B_{r}(x_0),$$
and
\begin{equation*}
\|\boldsymbol\varphi\|_{L^2(B_{r}(x_0))}+r\|D\boldsymbol\varphi\|_{L^2(B_{r}(x_0))}\leq Nr\|D_\ell (p\eta^2)+p\eta^2\Div\ell\|_{L^2(B_{r}(x_0))},
\end{equation*}
where  $N=N(d)$. Moreover, combining \eqref{defeta}, the local boundedness of $p$ in Lemma \ref{lemlocbdd} and \eqref{estDellk02}, we have
\begin{align}\label{estellpeta}
&\|\boldsymbol\varphi\|_{L^2(B_{r}(x_0))}+r\|D\boldsymbol\varphi\|_{L^2(B_{r}(x_0))}\nonumber\\
&\leq Nr\|D_\ell p\eta\|_{L^2(B_{r}(x_0))}+Nr\|p\eta D_\ell \eta\|_{L^2(B_{r}(x_0))}+Nr\|p\eta^2\Div\ell\|_{L^2(B_{r}(x_0))}\nonumber\\
&\leq Nr\|D_\ell p\eta\|_{L^2(B_{r}(x_0))}+N\mathcal{C}_2r^{d/2}.
\end{align}
Applying $\boldsymbol\varphi$ to \eqref{equell} as a test function, we have
\begin{align*}
\int_{B_{r}(x_0)}\eta^2|D_\ell p|^2\,dx
&=-\int_{B_{r}(x_0)}A^{\alpha\beta}D_\beta {\bf u}_{_\ell}D_\alpha\boldsymbol\varphi\,dx-\int_{B_{r}(x_0)}{\bf f}\boldsymbol\varphi\,dx+\int_{B_{r}(x_0)}{\bf f}^{\alpha,3} D_\alpha\boldsymbol\varphi\,dx\\
&\quad-\int_{B_{r}(x_0)}D_\ell p(2p\eta D_\ell\eta+p\eta^2\Div\ell)\,dx.
\end{align*}
By Young's inequality and \eqref{estellpeta}, we have
\begin{align*}
\|\eta D_\ell p\|_{L^2(B_{r}(x_0))}
&\leq \varepsilon_2\|\eta D_\ell p\|_{L^2(B_{r}(x_0))}+N(\varepsilon_2)\big(\|D{\bf u}_{_\ell}\|_{L^2(B_{r}(x_0))}+\|{\bf f}^{\alpha,3}\|_{L^2(B_{r}(x_0))}\\
&\quad+r\|{\bf f}\|_{L^2(B_{r}(x_0))}\big)+N\mathcal{C}_2r^{\frac{d}{2}-1}.
\end{align*}
Taking $\varepsilon_2=\frac{1}{2}$, and using $\eta=1$ in $B_{r/2}(x_0)$, \eqref{est-g1}--\eqref{estDuell}, we obtain
\begin{align*}
\|D_\ell p\|_{L^2(B_{r/2}(x_0))}
&\leq Nr^{\frac{d+1}{2}}\Big(\sum_{j=1}^{m+1}\|D^2{\bf u}\|_{L^\infty(B_{r}(x_0)\cap\cD_j)}+\sum_{j=1}^{m+1}\|Dp\|_{L^\infty(B_{r}(x_0)\cap\cD_j)}\Big)\nonumber\\
&\quad+N\mathcal{C}_2r^{\frac{d}{2}-1}.
\end{align*}
This together with \eqref{estauxiu} gives the estimate of $\|\tilde p\|_{L^2(B_{r/2}(x_0))}$. The lemma is proved.
\end{proof}

\begin{lemma}\label{lemma Dtildeu}
Let $\varepsilon\in(0,1)$ and $q\in(1,\infty)$. Suppose that $A^{\alpha\beta}$, ${\bf f}^\alpha$, and $g$ satisfy Assumption \ref{assump} with $s=1$. If $({\bf u},p)\in W^{1,q}(B_{1})^d\times L^q(B_1)$ is a weak solution to
\begin{align*}
\begin{cases}
D_\alpha (A^{\alpha\beta}D_\beta {\bf u})+Dp=D_{\alpha}{\bf f}^{\alpha}\\
\Div {\bf u}=g
\end{cases}\,\,\mbox{in }~B_{1},
\end{align*}
then we have
\begin{align*}
\sum_{j=1}^{m+1}\|D^2{\bf u}\|_{L^\infty(B_{1/4}\cap\overline\cD_j)}+\sum_{j=1}^{m+1}\|Dp\|_{L^\infty(B_{1/4}\cap\overline\cD_j)}\leq N\mathcal{C}_2,
\end{align*}
where $\mathcal{C}_2$ is defined in \eqref{defC3}, $N>0$ is a constant depending only on
$d,m,q,\nu,\varepsilon$, $|A|_{1,\delta;\overline{\cD_{j}}}$, and the $C^{2,\mu}$ norm of $h_j$.
\end{lemma}

\begin{proof}
For any $s\in(0,1)$, let $\mathbf q_{x_{0},s}^{k'},\mathbf Q_{x_{0},s}\in\mathbb R^d$  be chosen such that
$$\Phi(x_{0},s)=\left(\fint_{B_{s}(x_{0})}\big(|D_{\ell^{k'}}\tilde {\bf u}(x;x_0)-\mathbf q_{x_{0},s}^{k'}|^{\frac{1}{2}}+|\tilde {\bf U}(x;x_0)-\mathbf Q_{x_{0},s}|^{\frac{1}{2}}\big)\,dx\right)^{2},$$
where $k'=1,\ldots,d-1$. It follows from the triangle inequality that
\begin{align*}
|\mathbf q_{x_{0},s/2}^{k'}-\mathbf q_{x_{0},s}^{k'}|^{\frac{1}{2}}\leq|D_{\ell^{k'}}\tilde {\bf u}(x;x_0)-\mathbf q_{x_{0}, s/2}^{k'}|^{\frac{1}{2}}+|D_{\ell^{k'}}\tilde {\bf u}(x;x_0)-\mathbf q_{x_{0},s}^{k'}|^{\frac{1}{2}}
\end{align*}
and
\begin{align*}
|\mathbf Q_{x_{0},s/2}-\mathbf Q_{x_{0},s}|^{\frac{1}{2}}\leq|\tilde {\bf U}(x;x_0)-\mathbf Q_{x_{0},s/2}|^{\frac{1}{2}}+|\tilde {\bf U}(x;x_0)-\mathbf Q_{x_{0},s}|^{\frac{1}{2}}.
\end{align*}
Taking the average over $x\in B_{s/2}(x_{0})$ and then taking the square, we obtain
\begin{align*}
|\mathbf q_{x_{0},s/2}^{k'}-\mathbf q_{x_{0},s}^{k'}|+|\mathbf Q_{x_{0},s/2}-\mathbf Q_{x_{0},s}|\leq N(\Phi(x_{0},s/2)+\Phi(x_{0},s)).
\end{align*}
By iterating and using the triangle inequality, we derive
\begin{equation}\label{itera}
|\mathbf q_{x_{0},2^{-L}s}^{k'}-\mathbf q_{x_{0},s}^{k'}|+|\mathbf Q_{x_{0},2^{-L} s}-\mathbf Q_{x_{0},s}|\leq N\sum_{j=0}^{L}\Phi(x_{0},2^{-j} s).
\end{equation}
Using \eqref{def-tildeu} and \eqref{deftildeU}, we have
\begin{align}\label{Dellk'tildeu}
D_{\ell^{k'}}\tilde {\bf u}(x;x_0)=\ell_i^k\ell_j^{k'}D_iD_j{\bf u}+D_{\ell^{k'}}\ell_i D_i{\bf u}-\sum_{j=1}^{m+1}D_{\ell^{k'}}\tilde\ell_{i,j} D_i{\bf u}(P_jx_0)-D_{\ell^{k'}}{\boldsymbol{\mathfrak u}}
\end{align}
and
\begin{align}\label{tildeU}
&\tilde {\bf U}(x;x_0)=n^\alpha\Big(A^{\alpha\beta}D_\beta D_i {\bf u}\ell_i^k-D_{\ell}{\bf f}^\alpha+D_\ell A^{\alpha\beta}D_\beta {\bf u}-A^{\alpha\beta}D_\beta{\boldsymbol{\mathfrak u}}\nonumber\\
&\quad-\delta_{\alpha d}\sum_{j=1}^{m}\mathbbm{1}_{x^d>h_j(x')} (n^d_j(x'))^{-1}\tilde {\bf h}_j(x')-A^{\alpha\beta}\sum_{j=1}^{m+1}\mathbbm{1}_{_{\cD_j^c}}D_\beta \tilde\ell_{i,j}D_i{\bf u}(P_jx_0)\Big)+{\bf n}(D_\ell p-\pi).
\end{align}
Recalling the assumption that $D{\bf u}$ and $p$ are piecewise $C^1$, $A^{\alpha\beta}$ and ${\bf f}^\alpha$ are piecewise $C^{1,\delta}$, and using \eqref{estauxiu},  it follows that $D_{\ell^{k'}}\tilde {\bf u}(x;x_0),\tilde {\bf U}(x;x_0)\in C^0 (\cD_{\varepsilon}\cap\overline{{\cD}_{j}})$.
Taking $\rho=2^{-L} s$ in \eqref{est phi'}, we have
$$\lim_{L\rightarrow\infty}\Phi(x_{0},2^{-L} s)=0.$$
Thus, for any $x_0\in \cD_{\varepsilon}\cap{\cD}_{j}$,  we obtain
\begin{equation*}
\lim_{L\rightarrow\infty}\mathbf q_{x_{0},2^{-L}s}^{k'}=D_{\ell^{k'}}\tilde {\bf u}(x_{0};x_0),\quad \lim_{L\rightarrow\infty}\mathbf Q_{x_{0},2^{-L} s}=\tilde {\bf U}(x_{0};x_0).
\end{equation*}
Now taking $L\rightarrow\infty $ in \eqref{itera}, choosing $s=r/2$, and using Lemma \ref{lemma itera}, we have for $r\in(0,1/4)$, $k'=1,\ldots,d-1$, and $x_0\in \cD_{\varepsilon}\cap{{\cD}_{j}}$,
\begin{align}\label{diffDtildeu}
&|D_{\ell^{k'}}\tilde {\bf u}(x_{0};x_{0})-\mathbf q_{x_{0},r/2}^{k'}|+|\tilde {\bf U}(x_{0};x_{0})-\mathbf Q_{x_{0},r/2}|\nonumber\\
&\leq N\sum_{j=0}^{\infty}\Phi(x_{0},2^{-j-1}r)\leq N\Phi(x_{0},r/2)+N\mathcal{C}_0r^{\delta_{\mu}},
\end{align}
where  $\delta_{\mu}=\min\big\{\frac{1}{2},\mu,\delta\big\}$, and $\mathcal{C}_0$ is defined in \eqref{defC0}.
By averaging the inequality
\begin{align*}
|\mathbf q_{x_{0},r/2}^{k'}|+|\mathbf Q_{x_{0},r/2}|&\leq |D_{\ell^{k'}}\tilde {\bf u}(x;x_{0})-\mathbf q_{x_{0},r/2}^{k'}|+|\tilde {\bf U}(x;x_{0})-\mathbf Q_{x_{0},r/2}|+|D_{\ell^{k'}}\tilde {\bf u}(x;x_{0})|\\
&\quad+|\tilde {\bf U}(x;x_{0})|
\end{align*}
over $x\in B_{r/2}(x_{0})$ and then taking the square, we have
\begin{align*}
|\mathbf q_{x_{0},r/2}^{k'}|+|\mathbf Q_{x_{0},r/2}|&\leq N\Phi(x_{0},r/2)+N\left(\fint_{B_{r/2}(x_{0})}\big(|D_{\ell^{k'}}\tilde {\bf u}(x;x_{0})|^{\frac{1}{2}}+|\tilde {\bf U}(x;x_{0})|^{\frac{1}{2}}\big)\,dx\right)^{2}\\
&\leq Nr^{-d}\Big(\|D_{\ell^{k'}}\tilde {\bf u}(\cdot;x_{0})\|_{L^{1}(B_{r/2}(x_{0}))}+\|\tilde {\bf U}(\cdot;x_{0})\|_{L^{1}(B_{r/2}(x_{0}))}\Big).
\end{align*}
Therefore, combining \eqref{diffDtildeu} and the triangle inequality, we obtain
\begin{align}\label{Dx'Uz0}
&|D_{\ell^{k'}}\tilde {\bf u}(x_0;x_0)|+|\tilde {\bf U}(x_0;x_0)|\nonumber\\
&\leq Nr^{-d}\Big(\|D_{\ell^{k'}}\tilde {\bf u}(\cdot;x_0)\|_{L^{1}(B_{r/2}(x_0))}+\|\tilde {\bf U}(\cdot;x_0)\|_{L^{1}(B_{r/2}(x_0))}\Big)+N\mathcal{C}_0r^{\delta_{\mu}}.
\end{align}
By using H\"{o}lder's inequality and Lemma \ref{lemmaup}, we have
\begin{align*}
&\|D\tilde{\bf u}(\cdot;x_0)\|_{L^1(B_{r/2}(x_0))}+\|\tilde p\|_{L^1(B_{r/2}(x_0))}\nonumber\\
&\leq Nr^{d+\frac{1}{2}}\big(\sum_{j=1}^{m+1}\|D^2{\bf u}\|_{L^\infty(B_{r}(x_0)\cap\cD_j)}+\sum_{j=1}^{m+1}\|Dp\|_{L^\infty(B_{r}(x_0)\cap\cD_j)}\big)+N\mathcal{C}_2r^{d-1}.
\end{align*}
Recalling  \eqref{deftildeU} and \eqref{def-tildeu}, and using \eqref{estauxiu}, we have
\begin{align*}
&\|\tilde{\bf U}(\cdot;x_0)\|_{L^1(B_{r/2}(x_0))}\nonumber\\
&\leq\|D\tilde{\bf u}(\cdot;x_0)\|_{L^1(B_{r/2}(x_0))}+\|\tilde{\bf f}(\cdot;x_0)\|_{L^1(B_{r/2}(x_0))}+\|\tilde p\|_{L^1(B_{r/2}(x_0))}\nonumber\\
&\leq Nr^{d+\frac{1}{2}}\big(\sum_{j=1}^{m+1}\|D^2{\bf u}\|_{L^\infty(B_{r}(x_0)\cap\cD_j)}+\sum_{j=1}^{m+1}\|Dp\|_{L^\infty(B_{r}(x_0)\cap\cD_j)}\big)+N\mathcal{C}_2r^{d-1}.
\end{align*}
These estimates together with \eqref{Dx'Uz0} imply that
\begin{align}\label{estDtildeuU}
&|D_{\ell^{k'}}\tilde {\bf u}(x_0;x_0)|+|\tilde {\bf U}(x_0;x_0)|\nonumber\\
&\leq Nr^{\delta_{\mu}}\big(\sum_{j=1}^{m+1}\|D^2{\bf u}\|_{L^\infty(B_{r}(x_0)\cap\cD_j)}+\sum_{j=1}^{m+1}\|Dp\|_{L^\infty(B_{r}(x_0)\cap\cD_j)}\big)+N\mathcal{C}_2r^{-1}.
\end{align}

It follows from \eqref{stokes} that
\begin{equation}\label{Dalpbetau}
A^{\alpha\beta}D_{\alpha\beta}{\bf u}+Dp=D_\alpha {\bf f}^\alpha-D_\alpha A^{\alpha\beta}D_{\beta}{\bf u}
\end{equation}
and
\begin{equation}\label{Ddiv}
D(\Div{\bf u})=Dg,
\end{equation}
in $B_{1-\varepsilon}\cap\cD_j$, $j=1,\ldots,M$. To solve for $D^2{\bf u}$ and $Dp$, we need to show that the determinant of the coefficient matrix in \eqref{Dellk'tildeu}, \eqref{tildeU}, \eqref{Dalpbetau}, and \eqref{Ddiv} is not equal to 0. To this end,  let us define
$$y=\Lambda x,\quad{\bf v}(y)=\Lambda{\bf u}(x),\quad \pi(y)=p(x),\quad\mathcal{A}^{\alpha\beta}(y)=\Lambda\Lambda^{\alpha k}A^{ks}(x)\Lambda^{s\beta}\Gamma,$$
where $\Gamma=\Lambda^{-1}$, $\Lambda$ is the linear transformation from the coordinate system associated with $0$ to the coordinate system associated with the fixed point $x\in B_r(x_0)$, which is defined in Section \ref{secpreliminaries} (see p.\pageref{deftauk}). A direct calculation yields
\begin{equation}\label{nAalp}
n^\alpha A^{\alpha\beta}D_\beta D_i{\bf u}\ell_i^k+{\bf n}D_\ell p=\Gamma\mathcal{A}^{d\beta}D_{\beta}D_{k}{\bf v}+{\bf n}D_k\pi.
\end{equation}
Using the definitions of $\Lambda$ and ${\bf n}$ in Section \ref{secpreliminaries} (see p.\pageref{deftauk}), we have
$$\Lambda{\bf n}=(0,\dots,0,1)^{\top}=:{\bf e}_d.$$
Then \eqref{nAalp} becomes
\begin{equation*}
\Lambda\big(n^\alpha A^{\alpha\beta}D_\beta D_i{\bf u}\ell_i^k+{\bf n}D_\ell p\big)=\mathcal{A}^{d\beta}D_{\beta}D_{k}{\bf v}+{\bf e}_dD_k\pi.
\end{equation*}
Similarly, we obtain
\begin{equation*}
\ell_i^k\ell_j^{k'}D_iD_j{\bf u}=\Gamma D_{k}D_{k'}{\bf v}
\end{equation*}
and
\begin{equation*}
A^{\alpha\beta}D_{\alpha\beta}{\bf u}+Dp=\Gamma (\mathcal{A}^{\alpha\beta}D_{\alpha\beta}{\bf v}+D\pi),\quad D(\Div{\bf u})=D(\Div{\bf v})\Lambda.
\end{equation*}
Thus, in view of \eqref{Dellk'tildeu}, \eqref{tildeU}, \eqref{Dalpbetau}, and \eqref{Ddiv}, we obtain the equations for $D^2{\bf v}$ and $D\pi$ as follows:
\begin{align}\label{eq-D2v}
\begin{cases}
D_{k}D_{k'}{\bf v}=\mathcal{R}_1,\\
\mathcal{A}^{d\beta}D_{\beta}D_{k}{\bf v}+{\bf e}_dD_k\pi=\mathcal{R}_2,\\
\mathcal{A}^{\alpha\beta}D_{\alpha\beta}{\bf v}+D\pi=\mathcal{R}_3,\\
D(\Div{\bf v})=\mathcal{R}_4,
\end{cases}
\end{align}
where $k,k'=1,\dots,d-1$, $\mathcal{R}_m$,  $m=1,2,3,4$, is derived from the terms in \eqref{Dellk'tildeu}, \eqref{tildeU}, \eqref{Dalpbetau}, and \eqref{Ddiv}, respectively. It follows from the first and last equations  in \eqref{eq-D2v} that $D_{k}D_{k'}{\bf v}$ and $D_kD_dv^d$ are solved, where $k,k'=1,\dots,d-1$. If  we solve for $D_dD_iv^j$ and $D_d\pi$, $i=1,\dots,d$, $j=1,\dots,d-1$, and $(i,j)=(d,d)$, then $D_k\pi$ are obtained from the second equation in \eqref{eq-D2v}, $k=1,\dots,d-1$. For this, we rewrite the last three equations in \eqref{eq-D2v} as, $i=1,\dots,d-1,~k=1,\dots,d-1$,
\begin{align}\label{eq-D2vDpi}
\begin{cases}
\sum_{j=1}^{d-1}\mathcal{A}_{ij}^{dd}D_{d}D_{k}v^j=\widetilde{\mathcal{R}}_2^i,\\
\sum_{\beta,j=1}^{d-1}\mathcal{A}_{ij}^{d\beta}D_{d\beta}v^j+\sum_{\alpha,j=1}^{d-1}\mathcal{A}_{ij}^{\alpha d}D_{\alpha d}v^j+\mathcal{A}_{ij}^{dd}D_{d }^2v^j-\sum_{j=1}^{d-1}\mathcal{A}_{dj}^{dd}D_{d}D_{i}v^j=\widetilde{\mathcal{R}}_3^i,\\
\sum_{\beta,j=1}^{d-1}\mathcal{A}_{dj}^{d\beta}D_{d\beta}v^j+\sum_{\alpha,j=1}^{d-1}\mathcal{A}_{dj}^{\alpha d}D_{\alpha d}v^j+\mathcal{A}_{dj}^{dd}D_{d }^2v^j+D_d\pi=\widetilde{\mathcal{R}}_3^d,\\
D_dD_jv^j=\mathcal{R}_4^d,
\end{cases}
\end{align}
where
\begin{align*}
\widetilde{\mathcal{R}}_2^i=\mathcal{R}_2^i-\sum_{\beta=1}^{d-1}\mathcal{A}_{ij}^{d\beta}D_{\beta}D_{k}v^j-\mathcal{A}_{id}^{dd}D_{d}D_{k}v^d,
\end{align*}
\begin{align*}
\widetilde{\mathcal{R}}_3^i&=\mathcal{R}_3^i-\sum_{\alpha,\beta=1}^{d-1}\mathcal{A}_{ij}^{\alpha\beta}D_{\alpha\beta}v^j-\mathcal{R}_2^d+\sum_{\beta=1}^{d-1}\mathcal{A}_{dj}^{d\beta}D_{\beta}D_{i}v^j-\sum_{\beta=1}^{d-1}\mathcal{A}_{id}^{d\beta}D_{d\beta}v^d\\
&\quad-\sum_{\alpha=1}^{d-1}\mathcal{A}_{id}^{\alpha d}D_{\alpha d}v^d+\mathcal{A}_{dd}^{dd}D_{d}D_{i}v^d,
\end{align*}
\begin{align*}
\widetilde{\mathcal{R}}_3^d=\mathcal{R}_3^d-\sum_{\alpha,\beta=1}^{d-1}\mathcal{A}_{dj}^{\alpha\beta}D_{\alpha\beta}v^j-\sum_{\beta=1}^{d-1}\mathcal{A}_{dd}^{d\beta}D_{d\beta}v^d-\sum_{\alpha=1}^{d-1}\mathcal{A}_{dd}^{\alpha d}D_{\alpha d}v^d,
\end{align*}
and $\mathcal{R}_m^i$ is the $i$-th component of $\mathcal{R}_m$, $m=2,3,4$. A direct calculation yields the determinant of the coefficient matrix in \eqref{eq-D2vDpi} is $(\mathrm{cof}(A_{dd}^{dd}))^d\neq0$, where $\mathrm{cof}(A_{dd}^{dd})$ is the cofactor of $(A^{dd})$. This implies that $D_dD_iv^j$ and $D_d\pi$ can be solved by Cramer's rule and thus $D^2{\bf u}$ and  $Dp$.
Moreover, using \eqref{estDtildeuU} and \eqref{estauxiu}, we obtain
\begin{align}\label{D2bfuDp}
&|D^2{\bf u}(x_0)|+|Dp(x_0)|\notag\\
&\leq N\Big(|D_{\ell^{k'}}\tilde {\bf u}(x_0;x_0)|+|\tilde {\bf U}(x_0;x_0)|+|D{\boldsymbol{\mathfrak u}}(x_0)|\Big)\nonumber\\
&\quad+N\big(\sum_{j=1}^{m+1}|{\bf f}^\alpha|_{1,\delta;\overline{\cD_{j}}}+\sum_{j=1}^{m+1}|g|_{1,\delta; \overline{\cD_{j}}}+\|D{\bf u}\|_{L^{1}(B_1)}+\|p\|_{L^{1}(B_1)}\big)\nonumber\\
&\leq Nr^{\delta_{\mu}}\big(\sum_{j=1}^{m+1}\|D^2{\bf u}\|_{L^\infty(B_{r}(x_0)\cap\cD_j)}+\sum_{j=1}^{m+1}\|Dp\|_{L^\infty(B_{r}(x_0)\cap\cD_j)}\big)+N\mathcal{C}_2r^{-1}.
\end{align}
For any $x_1\in B_{1/4}$ and $r\in(0,1/4)$, by taking the supremum with respect to $x_0\in B_r(x_1)\cap \cD_j$, we have
\begin{align*}
&\sum_{j=1}^{m+1}\|D^2{\bf u}\|_{L^\infty(B_{r}(x_1)\cap\cD_j)}+\sum_{j=1}^{m+1}\|Dp\|_{L^\infty(B_{r}(x_1)\cap\cD_j)}\nonumber\\
&\leq Nr^{\delta_{\mu}}\big(\sum_{j=1}^{m+1}\|D^2{\bf u}\|_{L^\infty(B_{2r}(x_1)\cap\cD_j)}+\sum_{j=1}^{m+1}\|Dp\|_{L^\infty(B_{2r}(x_1)\cap\cD_j)}\big)+N\mathcal{C}_2r^{-1}.
\end{align*}
Applying an iteration argument (see, for instance, \cite[Lemma 3.4]{dx2019}), we conclude that
\begin{align*}
\sum_{j=1}^{m+1}\|D^2{\bf u}\|_{L^\infty(B_{1/4}\cap\overline\cD_j)}+\sum_{j=1}^{m+1}\|Dp\|_{L^\infty(B_{1/4}\cap\overline\cD_j)}\leq N\mathcal{C}_2.
\end{align*}
We finish the proof of the lemma.
\end{proof}

\section{Proof of Theorem \ref{Mainthm} with \texorpdfstring{$s=1$}{}}\label{prfprop}
In this section, we first estimate $|D_{\ell^{k'}}\tilde {\bf u}(x;x_{0})-D_{\ell^{k'}}\tilde {\bf u}(x;x_{1})|$ and $|\tilde {\bf U}(x;x_{0})-\tilde {\bf U}(x;x_{1})|$, where $x_0, x_1\in B_{1-\varepsilon}$. Then we establish an a priori estimate of the modulus of continuity of $(D_{\ell^{k'}}\tilde{\bf u}, \tilde {\bf U})$ by using the results in Sections \ref{auxilemma} and \ref{bddestimate}, which implies Theorem \ref{Mainthm} with $s=1$.

\begin{lemma}\label{lemma uU}
Let $\varepsilon\in(0,1)$ and $q\in(1,\infty)$. Suppose that $A^{\alpha\beta}$, ${\bf f}^\alpha$, and $g$ satisfy Assumption \ref{assump} with $s=1$. If $(\tilde{\bf u},\tilde p)$ is a weak solution to \eqref{eqtildeu}, then for any $x_0, x_1\in B_{1-\varepsilon}$, we have
\begin{align}\label{diffUz0z1}
|D_{\ell^{k'}}\tilde {\bf u}(x;x_{0})-D_{\ell^{k'}}\tilde {\bf u}(x;x_{1})|+|\tilde {\bf U}(x;x_{0})-\tilde {\bf U}(x;x_{1})|\leq N\mathcal{C}_2r,
\end{align}
where $\mathcal{C}_2$ is defined in \eqref{defC3}, $N$ depends on $d,m,q,\nu,\varepsilon$, $|A|_{1,\delta;\overline{\cD_{j}}}$, and the $C^{2,\mu}$ characteristic of $\cD_{j}$.
\end{lemma}

\begin{proof}
We first note that for any $x_0\in B_{1/8}\cap\overline\cD_{j_0}$ and $x_1\in B_{1/8}\cap\overline\cD_{j_1}$,  by using \eqref{Pjx} and $h_j\in C^{2,\mu}$,
\begin{equation*}
|P_jx_0-P_jx_1|\leq N|x_0-x_1|.
\end{equation*}
Combining with Lemma \ref{lemma Dtildeu}, we have
\begin{align}\label{DuPjx}
\left|D{\bf u}(P_jx_0)-D{\bf u}(P_jx_1)\right|\leq Nr\|D^2{\bf u}\|_{L^\infty(B_{1/4}\cap\overline\cD_j)}\leq N\mathcal{C}_2r.
\end{align}
By \eqref{eq-rmu}, one has
\begin{align*}
\begin{cases}
D_{\alpha}(\tilde A^{\alpha\beta}D_\beta \tilde{\boldsymbol{\mathfrak u}}_j)+D\tilde{\mathfrak \pi}_j=-D_{\alpha}(\mathbbm{1}_{_{\cD_j^c}}A^{\alpha\beta}D_\beta \tilde\ell_{i,j}(D_i{\bf u}(P_jx_0)-D_i{\bf u}(P_jx_1)))& \mbox{in}~B_1,\\
\Div\tilde{\boldsymbol{\mathfrak u}}_j=\mathbbm{1}_{_{\cD_j^c}}D \tilde\ell_{i,j}(D_i{\bf u}(P_jx_1)-D_i{\bf u}(P_jx_0))\\
\qquad\qquad+(\mathbbm{1}_{_{\cD_j^c}}D \tilde\ell_{i,j}(D_i{\bf u}(P_jx_0)-D_i{\bf u}(P_jx_1)))_{B_1}& \mbox{in}~B_1,\\
\tilde{\boldsymbol{\mathfrak u}}_j=0& \mbox{on}~\partial B_1,
\end{cases}
\end{align*}
where
$$\tilde{\boldsymbol{\mathfrak u}}_j:={\boldsymbol{\mathfrak u}}_j(x;x_{0})-{\boldsymbol{\mathfrak u}}_j(x;x_{1}),\quad\tilde{\pi}_j:=\pi_j(x;x_0)-\pi_j(x;x_1).$$
Then by using Lemma \ref{lemlocbdd}, \eqref{DuPjx}, and the fact that $\mathbbm{1}_{_{\cD_j^c}}D_\beta \tilde\ell_{i,j}$ is piecewise $C^\mu$, $i=1,\ldots,m+1$, we obtain
\begin{align*}
&|\tilde{\boldsymbol{\mathfrak u}}_j|_{1,\mu';\overline{\cD_i}\cap B_{1-\varepsilon}}\nonumber\\
&\leq N\|D\tilde{\boldsymbol{\mathfrak u}}_j\|_{L^1(B_{1})}+\|\pi_j\|_{L^1(B_{1})}+N\sum_{j=1}^{m+1}|\mathbbm{1}_{_{\cD_j^c}}A^{\alpha\beta}D_\beta \tilde\ell_{i,j}(D_i{\bf u}(P_jx_0)-D_i{\bf u}(P_jx_1))|_{\mu;B_1}\nonumber\\
&\quad+N|\mathbbm{1}_{_{\cD_j^c}}D\tilde\ell_{i,j}(D_i{\bf u}(P_jx_1)-D_i{\bf u}(P_jx_0))|_{\mu;B_1}\nonumber\\
&\leq N\|\mathbbm{1}_{_{\cD_j^c}}A^{\alpha\beta}D_\beta \tilde\ell_{i,j}(D_i{\bf u}(P_jx_0)-D_i{\bf u}(P_jx_1))\|_{L^q(B_{1})}+N\mathcal{C}_2r\nonumber\\
&\leq N\mathcal{C}_2r,
\end{align*}
where $\mu'=\min\{\mu,\frac{1}{2}\}$. Thus,
\begin{align*}
|{\boldsymbol{\mathfrak u}}(\cdot;x_{0})-{\boldsymbol{\mathfrak u}}(\cdot;x_{1})|_{1,\mu';\overline{\cD_i}\cap B_{1-\varepsilon}}\leq
\sum_{j=1}^{m+1}|\tilde{\boldsymbol{\mathfrak u}}_j|_{1,\mu';\overline{\cD_i}\cap B_{1-\varepsilon}}\leq N\mathcal{C}_2r.
\end{align*}
This combined with  \eqref{def-tildeu}, \eqref{defu0}, and \eqref{DuPjx} yields
\begin{align}\label{diffDuz0z1}
&|D_{\ell^{k'}}\tilde {\bf u}(x;x_{0})-D_{\ell^{k'}}\tilde {\bf u}(x;x_{1})|\nonumber\\
&=\Big|\sum_{j=1}^{m+1}D_{\ell^{k'}}\tilde{\ell}_{i,j}\big(D_i{\bf u}(P_jx_0)-D_i{\bf u}(P_jx_1)\big)
+D_{\ell^{k'}}{\boldsymbol{\mathfrak u}}(x;x_{0})-D_{\ell^{k'}}{\boldsymbol{\mathfrak u}}(x;x_{1})\Big|\leq N\mathcal{C}_2r.
\end{align}
Similarly, we have the estimate of $|\tilde {\bf U}(x;x_{0})-\tilde {\bf U}(x;x_{1})|$ and thus the proof of Lemma \ref{lemma uU} is finished.
\end{proof}

Together with the results in Sections \ref{auxilemma} and \ref{bddestimate}, we obtain an a priori estimate of the modulus of continuity of $(D_{\ell^{k'}}\tilde{\bf u}, \tilde {\bf U})$ as follows.
\begin{proposition}\label{proptildeu}
Let $\varepsilon\in(0,1)$ and $q\in(1,\infty)$. Suppose that $A^{\alpha\beta}$, ${\bf f}^\alpha$, and $g$ satisfy Assumption \ref{assump} with $s=1$. If $({\bf u},p)\in W^{1,q}(B_{1})^d\times L^q(B_1)$ is a weak solution to
$$\begin{cases}
D_\alpha (A^{\alpha\beta}D_\beta {\bf u})+Dp=D_{\alpha}{\bf f}^{\alpha}\\
\Div {\bf u}=g
\end{cases}\,\,\mbox{in }~B_1,
$$
then for any $x_0, x_1\in B_{1-\varepsilon}$, we have
\begin{align}\label{holdertildeuU}
|(D_{\ell^{k'}}\tilde {\bf u}(x_{0};x_{0})-D_{\ell^{k'}}\tilde {\bf u}(x_{1};x_{1})|+|\tilde {\bf U}(x_{0};x_{0})-\tilde {\bf U}(x_{1};x_{1})|\leq N\mathcal{C}_2|x_{0}-x_{1}|^{\delta_{\mu}},
\end{align}
where $k'=1,\ldots,d-1$, $\mathcal{C}_2$ is defined in \eqref{defC3}, $\tilde {\bf u}$ and $\tilde {\bf U}$ are defined in \eqref{def-tildeu} and \eqref{deftildeU}, respectively, $\delta_{\mu}=\min\big\{\frac{1}{2},\mu,\delta\big\}$, $N$ depends on $d,m,q,\nu,\varepsilon$, $|A|_{1,\delta;\overline{\cD_{j}}}$, and the $C^{2,\mu}$ characteristic of $\cD_{j}$.
\end{proposition}

\begin{proof}
It follows from \eqref{Dellk'tildeu}  that
\begin{align}\label{Dellk'tildeu00}
D_{\ell^{k'}}\tilde {\bf u}(x_0;x_0)=\ell_i^k(x_0)\ell_j^{k'}(x_0)D_iD_j{\bf u}(x_0)-\sum_{j=1,j\neq j_0}^{m+1}D_{\ell^{k'}}\tilde\ell_{i,j}(x_0) D_i{\bf u}(P_jx_0)-D_{\ell^{k'}}{\boldsymbol{\mathfrak u}}(x_0).
\end{align}
For any $x_{1}\in B_{1/8}\cap\cD_{j_{1}}$, where $j_{1}\in\{1,\ldots,m+1\}$, if $|x_{0}-x_{1}|\geq1/16$, then by using \eqref{Dellk'tildeu00}, Lemma \ref{lemlocbdd}, Lemma \ref{lemma Dtildeu}, and \eqref{estauxiu}, we have
\begin{align*}
|D_{\ell^{k'}}\tilde {\bf u}(x_{0};x_0)-D_{\ell^{k'}}\tilde {\bf u}(x_{1};x_1)|
&\leq N\sum_{j=1}^{m+1}\|D^2{\bf u}\|_{L^\infty(B_{1/4}\cap\overline\cD_j)}+N\|D{\boldsymbol{\mathfrak u}}\|_{L^\infty(B_{1/4})}+N\mathcal{C}_2\notag\\
&\leq N\mathcal{C}_2|x_{0}-x_{1}|^{\delta_{\mu}}.
\end{align*}
Similarly, by using \eqref{tildeU} and the equation \eqref{stokes}, we have
\begin{align}\label{tildeU00}
\tilde {\bf U}(x_0;x_0)&=n^\alpha(x_0)\Big(A^{\alpha\beta}(x_0)D_\beta D_i {\bf u}(x_0)\ell_i^k(x_0)-D_{\ell}{\bf f}^\alpha(x_0)+D_\ell A^{\alpha\beta}(x_0)D_\beta {\bf u}(x_0)\nonumber\\
&\quad-A^{\alpha\beta}(x_0)D_\beta{\boldsymbol{\mathfrak u}}(x_0)-\delta_{\alpha d}\sum_{j=1}^{m}\mathbbm{1}_{x^d>h_j(x')} (n^d_j(x'_0))^{-1}\tilde {\bf h}_j(x'_0)\nonumber\\
&\quad-A^{\alpha\beta}(x_0)\sum_{j=1,j\neq j_0}^{m+1}\mathbbm{1}_{_{\cD_j^c}}D_\beta \tilde\ell_{i,j}(x_0)D_i{\bf u}(P_jx_0)\Big)\nonumber\\
&\quad+{\bf n}(x_0)\ell(x_0)\big(D_\alpha {\bf f}^\alpha(x_0)-D_\alpha A^{\alpha\beta}D_{\beta}{\bf u}(x_0)-A^{\alpha\beta}(x_0)D_{\alpha\beta}{\bf u}(x_0)\big)\nonumber\\
&\quad-{\bf n}(x_0)\pi(x_0;x_0),
\end{align}
and thus,
\begin{align*}
|\tilde {\bf U}(x_{0};x_0)-\tilde {\bf U}(x_{1};x_1)|
\leq N\mathcal{C}_2|x_{0}-x_{1}|^{\delta_{\mu}}.
\end{align*}

If $|x_0-x_1|<1/16$, then we set $r=|x_0-x_1|$. By the triangle inequality, for any $x\in B_{r}(x_0)\cap B_{r}(x_1)$, we have
\begin{equation}\label{differenceDu}
\begin{split}
&|D_{\ell^{k'}}\tilde {\bf u}(x_0;x_0)-D_{\ell^{k'}}\tilde {\bf u}(x_1;x_1)|^{\frac{1}{2}}+|\tilde {\bf U}(x_0;x_0)-\tilde {\bf U}(x_1;x_1)|^{\frac{1}{2}}\\
&\leq|D_{\ell^{k'}}\tilde {\bf u}(x_0;x_0)-\mathbf q_{x_0,r}^{k'}|^{\frac{1}{2}}+|D_{\ell^{k'}}\tilde {\bf u}(x;x_0)-\mathbf q_{x_0,r}^{k'}|^{\frac{1}{2}}+|D_{\ell^{k'}}\tilde {\bf u}(x;x_1)-\mathbf q_{x_1,r}^{k'}|^{\frac{1}{2}}\\
&\quad+|D_{\ell^{k'}}\tilde {\bf u}(x;x_0)-D_{\ell^{k'}}\tilde {\bf u}(x;x_1)|^{\frac{1}{2}}+|D_{\ell^{k'}}\tilde {\bf u}(x_1;x_1)-\mathbf q_{x_1,r}^{k'}|^{\frac{1}{2}}\\
&\quad+|\tilde {\bf U}(x_0;x_0)-\mathbf Q_{x_0,r}|^{\frac{1}{2}}+|\tilde {\bf U}(x;x_0)-\mathbf Q_{x_0,r}|^{\frac{1}{2}}+|\tilde {\bf U}(x;x_1)-\mathbf Q_{x_1,r}|^{\frac{1}{2}}\\
&\quad+|\tilde {\bf U}(x;x_0)-\tilde {\bf U}(x;x_1)|^{\frac{1}{2}}+|\tilde {\bf U}(x_1;x_1)-\mathbf Q_{x_1,r}|^{\frac{1}{2}},
\end{split}
\end{equation}
where $\mathbf q_{x_{0},r}^{k'},\mathbf Q_{x_{0},r},\mathbf q_{x_{1},r}^{k'},\mathbf Q_{x_{1},r}\in\mathbb R^d$, $k'=1,\ldots,d-1$, satisfy	
$$
\Phi(x_{0},r)=\left(\fint_{B_{r}(x_{0})}\big(|D_{\ell^{k'}}\tilde {\bf u}(x;x_{0})-\mathbf q_{x_{0},r}^{k'}|^{\frac{1}{2}}+|\tilde {\bf U}(x;x_{0})-\mathbf Q_{x_{0},r}|^{\frac{1}{2}}\big)\,dx\right)^{2},
$$
and
$$
\Phi(x_{1},r)=\left(\fint_{B_{r}(x_{1})}\big(|D_{\ell^{k'}}\tilde {\bf u}(x;x_1)-\mathbf q_{x_{1},r}^{k'}|^{\frac{1}{2}}+|\tilde {\bf U}(x;x_1)-\mathbf Q_{x_{1},r}|^{\frac{1}{2}}\big)\,dx\right)^{2},
$$
respectively. Taking the average over $x\in B_{r}(x_{0})\cap B_r(x_1)$ and then taking the  square in \eqref{differenceDu}, we obtain
\begin{equation}\label{diffDtildeuU}
\begin{split}
&|D_{\ell^{k'}}\tilde {\bf u}(x_{0};x_0)-D_{\ell^{k'}}\tilde {\bf u}(x_{1};x_1)|+|\tilde {\bf U}(x_{0};x_0)-\tilde {\bf U}(x_{1};x_1)|\\
&\leq|D_{\ell^{k'}}\tilde {\bf u}(x_0;x_0)-\mathbf q_{x_{0},r}^{k'}|+|\tilde {\bf U}(x_{0};x_0)-\mathbf Q_{x_{0},r}|+\Phi(x_0,r)+\Phi(x_1,r)\\
&\quad+|D_{\ell^{k'}}\tilde {\bf u}(x_{1};x_1)-\mathbf q_{x_{1},r}^{k'}|+|\tilde {\bf U}(x_1;x_1)-\mathbf Q_{x_{1},r}|\\
&\quad+\left(\fint_{B_{r}(x_{0})\cap B_r(x_1)}\big(|D_{\ell^{k'}}\tilde {\bf u}(x;x_0)-D_{\ell^{k'}}\tilde {\bf u}(x;x_1)|^{\frac{1}{2}}+|\tilde {\bf U}(x;x_0)-\tilde {\bf U}(x;x_1)|^{\frac{1}{2}}\big)\,dx\right)^{2}.
\end{split}
\end{equation}
It follows from Lemmas \ref{lemma itera} and  \ref{lemma Dtildeu}, \eqref{estauxiu}, \eqref{estfintDtildeu}, and \eqref{estfinttildef}  with $B_{1/8}$ in place of $B_r(x_0)$ that
\begin{align}\label{estsupphi}
\sup_{x_{0}\in B_{1/8}}\Phi(x_{0},r)&\leq Nr^{\delta_{\mu}}\Big(\sum_{j=1}^{m+1}\|D^2{\bf u}\|_{L^\infty(B_{1/4}\cap\overline\cD_j)}+\sum_{j=1}^{m+1}\|Dp\|_{L^\infty(B_{1/4}\cap\overline\cD_j)}+\|D\tilde {\bf u}\|_{L^1(B_{1/4})}\nonumber\\
&\quad+\|\tilde {\bf f}\|_{L^1(B_{1/4})}+\|\tilde p\|_{L^1(B_{1/4})}+\sum_{j=1}^{m+1}|{\bf f}^\alpha|_{1,\delta; \overline{\cD_{j}}}+\sum_{j=1}^{m+1}|g|_{1,\delta; \overline{\cD_{j}}}+\|D{\bf u}\|_{L^{1}(B_1)}\nonumber\\
&\quad+\|p\|_{L^{1}(B_1)}\Big)\leq N\mathcal{C}_2r^{\delta_{\mu}}.
\end{align}
Applying \eqref{diffDtildeu}  and using \eqref{estsupphi}, we derive
\begin{align}\label{estDk'DdU}
\sup_{x_{0}\in B_{1/8}}\big(|D_{\ell^{k'}}\tilde {\bf u}(x_{0};x_{0})-\mathbf q_{x_{0},r}^{k'}|+|\tilde {\bf U}(x_{0};x_{0})-\mathbf Q_{x_{0},r}|\big)\leq N\mathcal{C}_2r^{\delta_{\mu}}.
\end{align}
Substituting \eqref{estsupphi}, \eqref{estDk'DdU},  \eqref{diffDuz0z1}, and \eqref{diffUz0z1} into \eqref{diffDtildeuU}, we obtain \eqref{holdertildeuU}.
\end{proof}

\begin{proof}[{\bf Proof of Theorem \ref{Mainthm} with $s=1$}]
By using \eqref{Dalpbetau} and \eqref{Ddiv} at the point $x=x_0$, \eqref{Dellk'tildeu00},  \eqref{tildeU00}, and Cramer's rule, we get that  $D^{2}{\bf u}(x_0)$ and $Dp(x_0)$ are combinations of
\begin{equation}\label{uteq}
Dg(x_0),\quad D_\alpha {\bf f}^\alpha(x_0)-D_\alpha A^{\alpha\beta}(x_0)D_{\beta}{\bf u}(x_0),
\end{equation}
\begin{equation}\label{Duell}
D_{\ell^{k'}}\tilde {\bf u}(x_0;x_0)+\sum_{j=1,j\neq j_0}^{m+1}D_{\ell^{k'}}\tilde\ell_{i,j}(x_0) D_i{\bf u}(P_jx_0)+D_{\ell^{k'}}{\boldsymbol{\mathfrak u}}(x_0),
\end{equation}
and
\begin{align}\label{tildeUeq}
&\tilde {\bf U}(x_0;x_0)+n^\alpha(x_0)\Big(D_{\ell}{\bf f}^\alpha(x_0)-D_\ell A^{\alpha\beta}(x_0)D_\beta {\bf u}(x_0)+A^{\alpha\beta}(x_0)D_\beta{\boldsymbol{\mathfrak u}}(x_0)\nonumber\\
&\,+\delta_{\alpha d}\sum_{j=1}^{m}\mathbbm{1}_{x^d>h_j(x')} (n^d_j(x'_0))^{-1}\tilde {\bf h}_j(x'_0)+A^{\alpha\beta}(x_0)\sum_{j=1,j\neq j_0}^{m+1}\mathbbm{1}_{_{\cD_j^c}}D_\beta \tilde\ell_{i,j}(x_0)D_i{\bf u}(P_jx_0)\Big)\nonumber\\
&\quad-{\bf n}(x_0)\ell(x_0)\big(D_\alpha {\bf f}^\alpha(x_0)-D_\alpha A^{\alpha\beta}D_{\beta}{\bf u}(x_0)\big)+{\bf n}(x_0)\pi(x_0;x_0).
\end{align}
Similarly, for any $\tilde x_{0}\in B_{1-\varepsilon}\cap\overline{\cD}_{j_0}$, $D^{2}{\bf u}(\tilde x_0)$ and $Dp(\tilde x_0)$ are combinations of \eqref{uteq}--\eqref{tildeUeq} with $x_0$ replaced with $\tilde x_0$. It follows from  \eqref{holdertildeuU} and \eqref{estauxiu} that
$$[D^2{\bf u}]_{\delta_{\mu};B_{1-\varepsilon}\cap\overline{\cD}_{j_0}}+[Dp]_{\delta_{\mu};B_{1-\varepsilon}\cap\overline{\cD}_{j_0}}
\le N\mathcal{C}_2$$
for any $j_0=1,\ldots,m+1$. Theorem \ref{Mainthm} is proved.
\end{proof}

\section{The case when \texorpdfstring{$s\geq2$}{}}\label{general}
\subsection{Main ingredients of the proof}
We first use an induction argument for $s\geq 2$ to obtain
\begin{align}\label{Dlu}
D_{\ell}\cdots D_{\ell}{\bf u}
=\ell_{i_1}\ell_{i_2}\cdots\ell_{i_{_{s}}}D_{i_1}D_{i_2}\cdots D_{i_{_{s}}}{\bf u}+R({\bf u}),
\end{align}
where we used $D_\ell(fg)=gD_\ell f+fD_\ell g$ and the Einstein summation convention over repeated indices, $\ell_{i_{\tau}}:=\ell_{i_{\tau}}^{k_{\tau}}$, $\tau=1,\ldots,s$, $k_{\tau}=1,\ldots,d-1$, $i_{\tau}=1,\ldots,d$, and
\begin{align*}
R({\bf u})&=D_{\ell_{i_1}}(\ell_{i_{2}}\cdots\ell_{i_{_{s}}})D_{i_{2}}\cdots D_{i_{_{s}}}{\bf u}\nonumber\\
&\quad+D_{\ell_{i_1}}\Bigg(D_{\ell_{i_2}}(\ell_{i_{3}}\cdots\ell_{i_{_{s}}})D_{i_{3}}\cdots D_{i_{_{s}}}{\bf u}+D_{\ell_{i_2}}\Big(D_{\ell_{i_3}}(\ell_{i_{4}}\cdots\ell_{i_{_{s}}})D_{i_{4}}\cdots D_{i_{_{s}}}{\bf u}\nonumber\\
&\qquad\qquad+D_{\ell_{i_3}}\big(D_{\ell_{i_4}}(\ell_{i_{5}}\cdots\ell_{i_{_{s}}})D_{i_{5}}\cdots D_{i_{_{s}}}u+\cdots+D_{\ell_{i_{s-2}}}(D_{\ell_{i_{s-1}}}\ell_{i_s}D_{i_s}{\bf u})\big)\Big)\Bigg),
\end{align*}
which is the summation of the products of directional derivatives of $\ell$ and derivatives of ${\bf u}$. Taking $D_\ell\cdots D_{\ell}$ to the equation $D_\alpha(A^{\alpha\beta}D_\beta {\bf u})+Dp=D_\alpha {\bf f}^\alpha$ and $\Div{\bf u}=g$, respectively, we obtain in $\bigcup_{j=1}^{m+1}\cD_j$,
\begin{align}\label{eqDllu}
\begin{cases}
D_\alpha\big(A^{\alpha\beta}D_\beta (D_{\ell}\cdots D_{\ell}{\bf u})\big)+D(D_{\ell}\cdots D_{\ell}p)=D_\alpha \breve {\bf f}^{\alpha,1}+\breve {\bf f},\\
\Div(D_\ell\cdots D_{\ell}{\bf u})=\ell_{i_1}\ell_{i_2}\cdots\ell_{i_{_{s}}}D_{i_1}D_{i_2}\cdots D_{i_{_{s}}}g+D_\alpha(R(u^\alpha))\\
\qquad\qquad\qquad\qquad+D_\alpha(\ell_{i_1}\ell_{i_2}\cdots\ell_{i_{_{s}}})D_{i_1}D_{i_2}\cdots D_{i_{_{s}}}u^\alpha,
\end{cases}
\end{align}
where $\breve {\bf f}^{\alpha,1}=(\breve {f}_1^{\alpha,1},\dots,\breve { f}_d^{\alpha,1})^\top$, $\breve {\bf f}=(\breve {f}_1,\dots,\breve {f}_d)^\top$,  for the $i$-th equation, $i=1,\dots,d$,
\begin{align}\label{defbrevef1}
\breve { f}_i^{\alpha,1}&:=\ell_{i_1}\ell_{i_2}\cdots\ell_{i_{_{s}}}D_{i_1}D_{i_2}\cdots D_{i_{_{s}}}{f}_i^\alpha+A_{ij}^{\alpha\beta}D_\beta(\ell_{i_1}\ell_{i_2}\cdots\ell_{i_{_{s}}})D_{i_1}D_{i_2}\cdots D_{i_{_{s}}}{u^j}\notag\\
&\quad +A_{ij}^{\alpha\beta}D_\beta(R({u^j}))+\delta_{\alpha i}R(p)-\ell_{i_1}\ell_{i_2}\cdots\ell_{i_{_{s}}}\big(D_{i_1}A_{ij}^{\alpha\beta}D_\beta D_{i_2}\cdots D_{i_{_{s}}}{u^j}\nonumber\\
&\qquad+\sum_{\tau=1}^{s-1}D_{i_1}\cdots D_{i_{\tau}}(D_{i_{\tau+1}}A_{ij}^{\alpha\beta}D_\beta D_{i_{\tau+2}}\cdots D_{i_{_{s}}}{u^j})\big),
\end{align}
and
\begin{align}\label{defbrevef}
\breve { f}_i&:=D_\alpha(\ell_{i_1}\ell_{i_2}\cdots\ell_{i_{_{s}}})\big(D_{i_1}D_{i_2}\cdots D_{i_{_{s}}}(A_{ij}^{\alpha\beta}D_\beta {u^j}-{f}_i^\alpha+\delta_{\alpha i}p)\big)\nonumber\\
&\quad+R(D_\alpha({f}_i^\alpha-A_{ij}^{\alpha\beta}D_\beta {u^j})-D_ip).
\end{align}
Similarly, by taking $D_{\ell}\cdots D_\ell$ to  $[n^\alpha_j(A^{\alpha\beta} D_\beta {\bf u} -{\bf f}^\alpha)+p{\bf n}_j]_{\Gamma_j}=0$, we obtain the boundary condition
\begin{equation}\label{boundary}
[n_j^\alpha (A^{\alpha\beta} D_\beta (D_{\ell}\cdots D_{\ell}{\bf u})- \breve {\bf f}^{\alpha,1})]_{\Gamma_j}=\breve {\bf h}_j,
\end{equation}
where
\begin{align*}
\breve {\bf h}_j&:=\Big[-\ell_{i_1}\ell_{i_2}\cdots\ell_{i_{_{s}}}\big(\sum_{\tau=1}^{s}D_{i_{\tau}}n_j^\alpha D_{i_1}\cdots D_{\tau_{s-1}}D_{i_{\tau+1}}\cdots D_{i_s}(A^{\alpha\beta}D_\beta {\bf u}-{\bf f}^\alpha)\nonumber\\
&\quad+\sum_{\tau=1}^{s}D_{i_{\tau}}{\bf n}_j D_{i_1}\cdots D_{\tau_{s-1}}D_{i_{\tau+1}}\cdots D_{i_s}p\nonumber\\
&\quad+\sum_{1\leq \tau_1<\tau_2\leq s}D_{i_{\tau_1}}D_{i_{\tau_2}}n_j^\alpha D_{i_1}\cdots D_{i_{\tau_1}-1}D_{i_{\tau_1}+1}\cdots D_{i_{\tau_2}-1}D_{i_{\tau_2}+1}\cdots D_{i_s}(A^{\alpha\beta}D_\beta {\bf u}-{\bf f}^\alpha)\nonumber\\
&\quad+\sum_{1\leq \tau_1<\tau_2\leq s}D_{i_{\tau_1}}D_{i_{\tau_2}}{\bf n}_j D_{i_1}\cdots D_{i_{\tau_1}-1}D_{i_{\tau_1}+1}\cdots D_{i_{\tau_2}-1}D_{i_{\tau_2}+1}\cdots D_{i_s}p\nonumber\\
&\quad+\cdots+D_{i_1}D_{i_2}\cdots D_{i_s}n_j^\alpha (A^{\alpha\beta}D_\beta {\bf u}-{\bf f}^\alpha)+D_{i_1}D_{i_2}\cdots D_{i_s}{\bf n}_jp\big)\Big]_{\Gamma_j}\notag\\
&\quad -[R(n_j^\alpha (A^{\alpha\beta} D_\beta {\bf u}- {\bf f}^\alpha))+R({\bf n}_jp)]_{\Gamma_j}.
\end{align*}
By adding a term
$$
\sum_{j=1}^{m} D_d(\mathbbm{1}_{x^d>h_j(x')} (n^d_j(x'))^{-1}\breve{\bf h}_j(x'))
$$
to the first equation in \eqref{eqDllu}, then \eqref{eqDllu} and \eqref{boundary} become
\begin{align}\label{eq Dbreveu}
\begin{cases}
D_\alpha\big(A^{\alpha\beta}D_\beta (D_{\ell}\cdots D_{\ell}{\bf u})\big)+D(D_{\ell}\cdots D_{\ell}p)=D_\alpha \breve {\bf f}^{\alpha,2}+\breve {\bf f},\\
\Div(D_\ell\cdots D_{\ell}{\bf u})=\ell_{i_1}\ell_{i_2}\cdots\ell_{i_{_{s}}}D_{i_1}D_{i_2}\cdots D_{i_{_{s}}}g+D_\alpha(R(u^\alpha))\\
\qquad\qquad\qquad\qquad+D_\alpha(\ell_{i_1}\ell_{i_2}\cdots\ell_{i_{_{s}}})D_{i_1}D_{i_2}\cdots D_{i_{_{s}}}u^\alpha,\\
[n_j^\alpha (A^{\alpha\beta} D_\beta (D_{\ell}\cdots D_{\ell}{\bf u})- \breve {\bf f}^{\alpha,2})]_{\Gamma_j}=0,
\end{cases}
\end{align}
where
\begin{align*}
\breve {\bf f}^{\alpha,2}:=\breve {\bf f}^{\alpha,1}+\delta_{\alpha d}\sum_{j=1}^m\mathbbm{1}_{x^d>h_j(x')} (n^d_j(x'))^{-1}\breve{\bf h}_j(x').
\end{align*}

As mentioned above \eqref{tildeu}, since $D_\beta(\ell_{i_1}\ell_{i_2}\cdots\ell_{i_{_{s}}})$ and $R(u^j)$ are singular at any point where two interfaces touch or are close to each other, we cannot prove the smallness of the mean oscillation of \eqref{defbrevef1}. To cancel out the singularity, we choose
\begin{align}\label{defu-0}
&{\bf u}_0:={\bf u}_0(x;x_0)\notag\\
&=\sum_{j=1}^{m+1}\tilde\ell_{i_{1},j}\tilde\ell_{i_{2},j}\cdots\tilde\ell_{i_{s},j}D_{i_1}D_{i_2}\cdots D_{i_{_{s}}}{\bf u}(P_jx_0)\nonumber\\
&\quad+\sum_{j=1}^{m+1}\sum_{\tau=1}^{s-1}D_{\tilde\ell_{i_1,j}}D_{\tilde\ell_{i_2,j}}\cdots D_{\tilde\ell_{i_\tau},j}(\tilde\ell_{i_{\tau+1},j}\cdots\tilde\ell_{i_{_{s}},j})\big(D_{i_{\tau+1}}\cdots D_{i_{_{s}}}{\bf u}(P_jx_0)\nonumber\\
&\quad+(x-x_0)\cdot DD_{i_{\tau+1}}\cdots D_{i_{_{s}}}{\bf u}(P_jx_0)\big)+\cdots\nonumber\\
&\quad+\sum_{j=1}^{m+1}(D_{\tilde\ell_{i_{s-1},j}}\tilde \ell_{i_s,j})\tilde\ell_{i_{1},j}\tilde\ell_{i_{2},j}\cdots\tilde\ell_{i_{s-2},j}\big(D_{i_1}D_{i_2}\cdots D_{i_{_{s-2}}}D_{i_{_{s}}}{\bf u}(P_jx_0)\nonumber\\
&\quad+(x-x_0)\cdot DD_{i_{1}}D_{i_{2}}\cdots D_{i_{_{s-2}}}D_{i_{_{s}}}{\bf u}(P_jx_0)\big),
\end{align}
where $P_jx_0$ is defined in \eqref{Pjx},
$x_{0}\in B_{3/4}\cap \cD_{j_{0}}$, $r\in(0,1/4)$, $\tilde\ell_{,j}$ is the smooth extension of $\ell|_{\cD_j}$ to $\cup_{k=1,k\neq j}^{m+1}\cD_k$. Denote
\begin{equation}\label{defuell}
{\bf u}^{\ell}:={\bf u}^{\ell}(x;x_0)=D_{\ell}\cdots D_{\ell}{\bf u}-{\bf u}_0.
\end{equation}
Then by using \eqref{eq Dbreveu},  we obtain
\begin{align}\label{homosecond2}
\begin{cases}
D_\alpha(A^{\alpha\beta}D_\beta {\bf u}^{\ell})+DD_{\ell}\cdots D_{\ell}p=D_\alpha \breve{\bf f}^{\alpha,3}+\breve{\bf f}, \\
[n_j^\alpha (A^{\alpha\beta} D_\beta {\bf u}^{\ell}-{\bf f}^{\alpha,3})+{\bf n}_jD_{\ell}\cdots D_{\ell}p]_{\Gamma_j}=0,\\
\Div {\bf u}^{\ell}=\ell_{i_1}\ell_{i_2}\cdots\ell_{i_{_{s}}}D_{i_1}D_{i_2}\cdots D_{i_{_{s}}}g+D_\alpha(R(u^\alpha))-\Div {\bf u}_0\\
\qquad\qquad+D_\alpha(\ell_{i_1}\ell_{i_2}\cdots\ell_{i_{_{s}}})D_{i_1}D_{i_2}\cdots D_{i_{_{s}}}u^\alpha,
\end{cases}
\end{align}
where $\breve{\bf f}^{\alpha,3}=(\breve f_1^{\alpha,3},\dots,\breve f_d^{\alpha,3})^{\top}$, and
\begin{align}\label{brevef}
\breve f_i^{\alpha,3}&:=\breve f_i^{\alpha,3}(x;x_0)=\breve f_i^{\alpha,2}-A_{ij}^{\alpha\beta}D_\beta u_0^j,\quad i=1,\dots,d.
\end{align}

Finally, we consider the following problem:
\begin{align}\label{eqmathsfu}
\begin{cases}
D_\alpha(\tilde A^{\alpha\beta}D_\beta {\mathsf u})+D\pi=-D_{\alpha}\big(A^{\alpha\beta}{\bf F}_\beta\big)\\
\Div{\mathsf u}=-{\bf E}+({\bf E})_{B_1}
\end{cases}\,\, \mbox{in}~B_1,
\end{align}
where $({\mathsf u}(\cdot;x_0),\pi(\cdot;x_0))\in W_0^{1,q}(B_1)^d\times L_0^q(B_1)$, the coefficient $\tilde A^{\alpha\beta}$ is defined in \eqref{mathcalA},
\begin{align}\label{defmathF}
{\bf F}_\beta&:=\sum_{j=1}^{m+1}\mathbbm{1}_{_{\cD_j^c}}D_\beta (\tilde\ell_{i_{1},j}\tilde\ell_{i_{2},j}\cdots\tilde\ell_{i_{s},j})D_{i_1}D_{i_2}\cdots D_{i_{_{s}}}{\bf u}(P_jx_0)+\cdots\nonumber\\
&\quad+\sum_{j=1}^{m+1}\mathbbm{1}_{_{\cD_j^c}}D_\beta\big((D_{\tilde\ell_{i_{s-1},j}}\tilde \ell_{i_s,j})\tilde\ell_{i_{1},j}\tilde\ell_{i_{2},j}\cdots\tilde\ell_{i_{s-2},j}\big)\big(D_{i_1}D_{i_2}\cdots D_{i_{_{s-2}}}D_{i_{_{s}}}{\bf u}(P_jx_0)\nonumber\\
&\quad\quad+(x-x_0)\cdot DD_{i_{1}}D_{i_{2}}\cdots D_{i_{_{s-2}}}D_{i_{_{s}}}{\bf u}(P_jx_0)\big)\notag\\
&\quad+\sum_{j=1}^{m+1}\mathbbm{1}_{_{\cD_j^c}}
(D_{\tilde\ell_{i_{s-1},j}}\tilde \ell_{i_s,j})\tilde\ell_{i_{1},j}\tilde\ell_{i_{2},j}\cdots
\tilde\ell_{i_{s-2},j} D_\beta D_{i_{1}}D_{i_{2}}\cdots D_{i_{_{s-2}}}D_{i_{_{s}}}{\bf u}(P_jx_0),
\end{align}
which is the summation of the products of $\mathbbm{1}_{_{\cD_j^c}}$ and derivatives of the terms on the right-hand side of \eqref{defu-0},
and
\begin{align*}
{\bf E}&:=\sum_{j=1}^{m+1}\mathbbm{1}_{_{\cD_j^c}}D (\tilde\ell_{i_{1},j}\tilde\ell_{i_{2},j}\cdots\tilde\ell_{i_{s},j})D_{i_1}D_{i_2}\cdots D_{i_{_{s}}}{\bf u}(P_jx_0)+\cdots\nonumber\\
&\quad+\sum_{j=1}^{m+1}\mathbbm{1}_{_{\cD_j^c}}D\big((D_{\tilde\ell_{i_{s-1},j}}\tilde \ell_{i_s,j})\tilde\ell_{i_{1},j}\tilde\ell_{i_{2},j}\cdots\tilde\ell_{i_{s-2},j}\big)\big(D_{i_1}D_{i_2}\cdots D_{i_{_{s-2}}}D_{i_{_{s}}}{\bf u}(P_jx_0)\nonumber\\
&\quad\quad+(x-x_0)\cdot DD_{i_{1}}D_{i_{2}}\cdots D_{i_{_{s-2}}}D_{i_{_{s}}}{\bf u}(P_jx_0)\big)\notag\\
&\quad+\sum_{j=1}^{m+1}\mathbbm{1}_{_{\cD_j^c}}
(D_{\tilde\ell_{i_{s-1},j}}\tilde \ell_{i_s,j})\tilde\ell_{i_{1},j}\tilde\ell_{i_{2},j}\cdots
\tilde\ell_{i_{s-2},j} DD_{i_{1}}D_{i_{2}}\cdots D_{i_{_{s-2}}}D_{i_{_{s}}}{\bf u}(P_jx_0).
\end{align*}
Define
\begin{equation}\label{defbreveu}
\breve {\bf u}:=\breve {\bf u}(x;x_0)={\bf u}^{\ell}-{\mathsf u},\quad\breve p:=\breve p(x;x_0)=D_{\ell}\cdots D_{\ell}p-\pi.
\end{equation}
Then it follows from \eqref{homosecond2} and \eqref{eqmathsfu} that in $B_{3/4}$, $\breve {\bf u}$ and $\breve p$ satisfy
\begin{align}\label{eqbreveu}
\begin{cases}
D_\alpha(A^{\alpha\beta}D_\beta \breve {\bf u})+D\breve p=D_\alpha \breve {\bf f}^\alpha+\breve{\bf f},\\
\Div\breve {\bf u}=\ell_{i_1}\ell_{i_2}\cdots\ell_{i_{_{s}}}D_{i_1}D_{i_2}\cdots D_{i_{_{s}}}g-\Div {\bf u}_0+D_\alpha(\ell_{i_1}\ell_{i_2}\cdots\ell_{i_{_{s}}})D_{i_1}D_{i_2}\cdots D_{i_{_{s}}}u^\alpha\\
\qquad\qquad+{\bf E}-({\bf E})_{B_1},
\end{cases}
\end{align}
where $\breve {\bf f}^\alpha=(\breve f_1^\alpha,\dots,\breve f_d^\alpha)^{\top}$, and for $i=1,\dots,d$,
\begin{align}\label{def-brevefalpha}
\breve f_i^\alpha:=\breve f_i^\alpha(x;x_0)=\breve f_i^{\alpha,1}+\delta_{\alpha d}\sum_{j=1}^m\mathbbm{1}_{x^d>h_j(x')} (n^d_j(x'))^{-1}\breve h_j^i(x')-A_{ij}^{\alpha\beta}D_\beta u_0^j+A_{ij}^{\alpha\beta}F_\beta^j,
\end{align}
and $\breve f_i^{\alpha,1}$ is defined in \eqref{defbrevef1}.

The general case $s\geq2$ will be proved by induction on $s$. If $A^{\alpha\beta}$, ${\bf f}^{\alpha}$, and $g$ are piecewise $C^{s-1,\delta}$, and the interfacial boundaries are $C^{s,\mu}$, then we have
\begin{align}\label{estinduction}
&|{\bf u}|_{s,\delta_{\mu};\cD_{\varepsilon}\cap\overline{{\cD}_{j_0}}}+|p|_{s-1,\delta_{\mu};\cD_{\varepsilon}\cap\overline{{\cD}_{j_0}}}\nonumber\\
&\leq N\Big(\|D{\bf u}\|_{L^{1}(\cD)}+\|p\|_{L^{1}(\cD)}+\sum_{j=1}^{M}|{\bf f}^\alpha|_{s-1,\delta;\overline{\cD_{j}}}+\sum_{j=1}^{M}|g|_{s-1,\delta;\overline{\cD_{j}}}\Big),
\end{align}
where $j_0=1,\dots,m+1$, $\delta_{\mu}=\min\big\{\frac{1}{2},\mu,\delta\big\}$, and $N$ depends on $d,m,q,\nu,\varepsilon$, the $C^{s,\mu}$ characteristic of $\cD_{j}$, and $|A|_{s-1+\delta;\overline{\cD_{j}}}$. Now assuming that $A^{\alpha\beta}$, ${\bf f}^{\alpha}$, and $g$ are piecewise $C^{s,\delta}$, and the interfacial boundaries are $C^{s+1,\mu}$, we will prove that ${\bf u}$ is piecewise $C^{s+1,\delta_\mu}$ and $p$ is piecewise $C^{s,\delta_\mu}$.

Recalling that $\tilde\ell_{,j}$ is the smooth extension of $\ell|_{\cD_j}$ to $\cup_{k=1,k\neq j}^{m+1}\cD_k$ and using \eqref{estinduction}, one can see that the right-hand side of \eqref{eqmathsfu} is piecewise $C^{\delta_{\mu}}$. Then by applying Lemma \ref{lemlocbdd} to \eqref{eqmathsfu}, we have
\begin{align}\label{est-Dmathu}
&|\mathsf u|_{1+\delta_\mu;\overline{\cD_i}\cap B_{1-\varepsilon}}+|\pi|_{\delta_\mu;\overline{\cD_i}\cap B_{1-\varepsilon}}\nonumber\\
&\leq N\big(\|D{\bf u}\|_{L^{1}(\cD)}+\|p\|_{L^{1}(\cD)}+\sum_{j=1}^{M}|{\bf f}^\alpha|_{s-1,\delta;\overline{\cD_{j}}}+\sum_{j=1}^{M}|g|_{s-1,\delta;\overline{\cD_{j}}}\big),
\end{align}
where $i=1,\dots,m+1$. Therefore, combining with \eqref{defbreveu}, to derive the regularity of $D_{\ell}\cdots D_{\ell}{\bf u}$ and $D_{\ell}\cdots D_{\ell}p$, it suffices to prove that for $\breve {\bf u}$ and $\breve p$. For this, by replicating the argument in the proof of Lemma \ref{lemma itera}, we obtain the decay estimate of $\Psi(x_0,r)$ as follows, where
\begin{equation*}
\Psi(x_0,r):=\inf_{\mathbf q^{k'},\mathbf Q\in\mathbb R^{d}}\left(\fint_{B_r(x_0)}\big(|D_{\ell^{k'}}\breve {\bf u}(x;x_0)-\mathbf q^{k'}|^{\frac{1}{2}}+|\breve {\bf U}(x;x_0)-\mathbf Q|^{\frac{1}{2}}\big)\,dx \right)^{2},
\end{equation*}
and
\begin{equation}\label{defbreveU}
\breve {\bf U}(x;x_0)=n^\alpha(A^{\alpha\beta}D_\beta \breve {\bf u}-\breve {\bf f}^\alpha)+{\bf n}\breve p.
\end{equation}

\begin{lemma}\label{lemmaiter}
Let $\varepsilon\in(0,1)$ and $q\in(1,\infty)$. Suppose that $A^{\alpha\beta}$, ${\bf f}^\alpha$, and $g$ satisfy Assumption \ref{assump} with $s\geq2$. If $(\breve{\bf u},\breve p)$ is a weak solution to \eqref{eqbreveu}, then for any $0<\rho\leq r\leq 1/4$, we have
\begin{align*}
\Psi(x_{0},\rho)&\leq N\Big(\frac{\rho}{r}\Big)^{\delta_{\mu}}\Psi(x_{0},r/2)+N\mathcal{C}_3\rho^{\delta_{\mu}},
\end{align*}
where
\begin{align*}
\mathcal{C}_3&:=\sum_{j=1}^{m+1}\|D^{s+1}{\bf u}\|_{L^\infty(B_{r}(x_0)\cap\cD_j)}+\sum_{j=1}^{m+1}\|D^sp\|_{L^\infty(B_{r}(x_0)\cap\cD_j)}+\mathcal{C}_4,
\end{align*}
\begin{align}\label{defC4}
\mathcal{C}_4:=\|D{\bf u}\|_{L^{1}(B_1)}+\|p\|_{L^{1}(B_1)}+\sum_{j=1}^{M}|{\bf f}^\alpha|_{s,\delta;\overline{\cD_{j}}}+\sum_{j=1}^{M}|g|_{s,\delta; \overline{\cD_{j}}},
\end{align}
$\delta_{\mu}=\min\big\{\frac{1}{2},\mu,\delta\big\}$, $N$ depends on $d,m,q,\nu$, the $C^{s+1,\mu}$ norm of $h_j$, and $|A|_{s,\delta;\overline{\cD_{j}}}$.
\end{lemma}

By the definitions of $\breve{\bf f}$, ${\bf u}_0$, and $\breve{\bf f}^{\alpha,3}$ in \eqref{defbrevef}, \eqref{defu-0}, and\eqref{brevef},  respectively, using \eqref{estDellk02}, and  mimicking the proof of Lemma \ref{lemmaup}, we obtain the following result.

\begin{lemma}\label{lemmbddup-s}
Under the same assumptions as in Lemma \ref{lemmaiter}, we have
\begin{align*}
&\|D\breve{\bf u}(\cdot;x_0)\|_{L^2(B_{r/2}(x_0))}+\|\breve p(\cdot;x_0)\|_{L^2(B_{r/2}(x_0))}\nonumber\\
&\leq Nr^{\frac{d+1}{2}}\Big(\sum_{j=1}^{m+1}\|D^{s+1}{\bf u}\|_{L^\infty(B_{r}(x_0)\cap\cD_j)}+\sum_{j=1}^{m+1}\|D^sp\|_{L^\infty(B_{r}(x_0)\cap\cD_j)}\Big)+N\mathcal{C}_4r^{\frac{d}{2}-1},
\end{align*}
where $x_0\in \cD_{\varepsilon}\cap{{\cD}_{j_0}}$, $r\in(0,1/4)$,  $\breve{\bf u}$ and $\breve p$ are defined in \eqref{defbreveu}, the constant $N>0$ depends on $d,m,q,\nu,\varepsilon$, $|A|_{s,\delta;\overline{\cD_{j}}}$, and the $C^{s+1,\mu}$ norm of $h_j$.
\end{lemma}

\begin{lemma}\label{lemma Dbreveu}
Under the same assumptions as in Lemma \ref{lemmaiter}, if $({\bf u},p)\in W^{1,q}(B_{1})^d\times L^q(B_1)$ is a weak solution to
\begin{align*}
\begin{cases}
D_\alpha (A^{\alpha\beta}D_\beta {\bf u})+Dp=D_{\alpha}{\bf f}^{\alpha}\\
\Div {\bf u}=g
\end{cases}\,\,\mbox{in }~B_{1},
\end{align*}
then we have
\begin{align*}
\sum_{j=1}^{m+1}\|D^{s+1}{\bf u}\|_{L^\infty(B_{1/4}\cap\overline\cD_j)}+\sum_{j=1}^{m+1}\|D^{s}p\|_{L^\infty(B_{1/4}\cap\overline\cD_j)}\leq N\mathcal{C}_4,
\end{align*}
where $\mathcal{C}_4$ is defined in \eqref{defC4}, $N>0$ is a constant depending on
$d,m,q,\nu,\varepsilon$, $|A|_{s,\delta;\overline{\cD_{j}}}$, and the $C^{s+1,\mu}$ norm of $h_j$.
\end{lemma}

\begin{proof}
The proof is similar to that of Lemma \ref{lemma Dtildeu}.  It follows from \eqref{Dlu}, \eqref{defuell}, \eqref{defbreveu}, and \eqref{def-brevefalpha} that
\begin{align}\label{Dbreveu}
D_{\ell^{k}}\breve {\bf u}(x;x_0)
&=\ell_{i_1}\ell_{i_2}\cdots\ell_{i_{_{s}}}D_{\ell^k}D_{i_1}D_{i_2}\cdots D_{i_{_{s}}}{\bf u}+D_{\ell^k}(\ell_{i_1}\ell_{i_2}\cdots\ell_{i_{_{s}}})D_{i_1}D_{i_2}\cdots D_{i_{_{s}}}{\bf u}\nonumber\\
&\quad+D_{\ell^k}(R({\bf u}))-D_{\ell^{k}} {\bf u}_0-D_{\ell^{k}}{\mathsf u}
\end{align}
and
\begin{align}\label{breveU}
&\breve {\bf U}(x;x_0)=n^\alpha\big(A^{\alpha\beta}D_\beta \breve {\bf u}-\breve {\bf f}^\alpha\big)+{\bf n}\breve p\nonumber\\
&=n^\alpha\big(A^{\alpha\beta}\ell_{i_1}\ell_{i_2}\cdots\ell_{i_{_{s}}}D_\beta D_{i_1}D_{i_2}\cdots D_{i_{_{s}}}{\bf u}-A^{\alpha\beta}D_\beta{\mathsf u}-\ell_{i_1}\ell_{i_2}\cdots\ell_{i_{_{s}}}D_{i_1}D_{i_2}\cdots D_{i_{_{s}}}{\bf f}^\alpha\nonumber\\
&\quad+\ell_{i_1}\ell_{i_2}\cdots\ell_{i_{_{s}}}\big(D_{i_1}A^{\alpha\beta}D_\beta D_{i_2}\cdots D_{i_{_{s}}}{\bf u}+\sum_{\tau=1}^{s-1}D_{i_1}\cdots D_{i_{\tau}}(D_{i_{\tau+1}}A^{\alpha\beta}D_\beta D_{i_{\tau+2}}\cdots D_{i_{_{s}}}{\bf u})\big)\nonumber\\
&\quad-\delta_{\alpha d}\sum_{j=1}^m\mathbbm{1}_{x^d>h_j(x')} (n^d_j(x'))^{-1}\breve {\bf h}_j(x')-A^{\alpha\beta}{\bf F}_\beta\big)+{\bf n}(\ell_{i_1}\ell_{i_2}\cdots\ell_{i_{_{s}}}D_{i_1}D_{i_2}\cdots D_{i_{_{s}}}p-\pi).
\end{align}
Then using Lemmas \ref{lemmaiter}, \ref{lemmbddup-s}, and the argument that led to \eqref{estDtildeuU}, we have
\begin{align}\label{estDbreveU}
&|D_{\ell^{k'}}\breve {\bf u}(x_0;x_0)|+|\breve {\bf U}(x_0;x_0)|\nonumber\\
&\leq Nr^{\delta_{\mu}}\big(\sum_{j=1}^{m+1}\|D^{s+1}{\bf u}\|_{L^\infty(B_{r}(x_0)\cap\cD_j)}+\sum_{j=1}^{m+1}\|D^{s}p\|_{L^\infty(B_{r}(x_0)\cap\cD_j)}\big)+N\mathcal{C}_4r^{-1}.
\end{align}
Note that $D^{s+1}{\bf u}$ and $D^sp$ have  $d\tbinom{d+s}{s+1}$ and $\tbinom{d+s-1}{s}$ components, respectively. To solve for them, we first take the $(s-1)$-th derivative of the first equation \eqref{stokes} in each subdomain to get the following $d\tbinom{d+s-2}{s-1}$ equations
\begin{align}\label{eqDDDu}
A^{\alpha\beta}D_{\alpha\beta} D^{s-1}{\bf u}+D^{s}p=D^{s-1}D_\alpha {\bf f}^{\alpha}-\sum_{i=1}^{s-1}\tbinom{s-1}{i}D^i A^{\alpha\beta}D^{s-1-i}D_{\alpha\beta} {\bf u}-D^{s-1}(D_\alpha A^{\alpha\beta}D_\beta {\bf u}).
\end{align}
Here, it follows from \eqref{estinduction}, the assumption on $A^{\alpha\beta}$  and ${\bf f}^{\alpha}$ in Assumption \ref{assump}  that the right-hand side of \eqref{eqDDDu} is of class piecewise $C^{\delta_{\mu}}$.
Next, by taking the $s$-th derivative of the second equation \eqref{stokes} in each subdomain,  we obtain $\tbinom{d+s-1}{s}$ equations
\begin{equation}\label{eqDsduv}
D^s(\Div{\bf u})=D^sg.
\end{equation}
Finally, by the $d\tbinom{d+s-1}{s+1}+d\tbinom{d+s-2}{s}$ equations  in \eqref{Dbreveu} and \eqref{breveU}, and using \eqref{eqDDDu}, \eqref{eqDsduv}, and Cramer's rule, we derive $D^{s+1}{\bf u}$ and $D^sp$. Furthermore, combining \eqref{est-Dmathu} and \eqref{estDbreveU}, we obtain
\begin{align*}
&|D^{s+1}{\bf u}(x_0)|+|D^{s}p(x_0)|\\
&\leq Nr^{\delta_{\mu}}\big(\sum_{j=1}^{m+1}\|D^{s+1}{\bf u}\|_{L^\infty(B_{r}(x_0)\cap\cD_j)}+\sum_{j=1}^{m+1}\|D^{s}p\|_{L^\infty(B_{r}(x_0)\cap\cD_j)}\big)
+N\mathcal{C}_4r^{-1}.
\end{align*}
Finally, following the argument below \eqref{D2bfuDp}, Lemma \ref{lemma Dbreveu} is proved.
\end{proof}

\subsection{Proof of Theorem \ref{Mainthm} with \texorpdfstring{$s\geq2$}{}}
Using Lemmas \ref{lemmaiter} --\ref{lemma Dbreveu}, and following the argument in the proof of \eqref{holdertildeuU}, we reach an a priori estimate of the modulus of continuity of $(D_{\ell^{k'}}\breve {\bf u},\breve {\bf U})$ as follows:
\begin{align}\label{piecebreveu}
|(D_{\ell^{k'}}\breve {\bf u}(x_{0};x_{0})-D_{\ell^{k'}}\breve {\bf u}(x_{1};x_{1})|+|\breve {\bf U}(x_{0};x_{0})-\breve {\bf U}(x_{1};x_{1})|\leq N\mathcal{C}_4|x_{0}-x_{1}|^{\delta_{\mu}},
\end{align}
where $\mathcal{C}_4$ is defined by \eqref{defC4}, $x_0, x_1\in B_{1-\varepsilon}$, $k'=1,\ldots,d-1$, $\breve {\bf u}$ and $\breve {\bf U}$ are defined in \eqref{defbreveu} and \eqref{defbreveU}, respectively, $\delta_{\mu}=\min\big\{\frac{1}{2},\mu,\delta\big\}$, $N$ depends on $d,m,q,\nu,\varepsilon$, $|A|_{s,\delta;\overline{\cD_{j}}}$, and the $C^{s+1,\mu}$ characteristic of $\cD_{j}$.

For any $x_0\in B_{1-\varepsilon}\cap\overline{\cD}_{j_0}$, it follows from \eqref{Dlu} and \eqref{defu-0} that the terms containing (directional) derivatives of $\ell$ at $x_0$ in  \eqref{Dbreveu}  are cancelled. Then using \eqref{Dbreveu}, \eqref{breveU}, \eqref{eqDDDu}, and \eqref{eqDsduv} with $x=x_0$ and Cramer's rule, one can solve for $D^{s+1}{\bf u}(x_0)$ and $D^sp(x_0)$. For any $x_1\in B_{1-\varepsilon}\cap\overline{\cD}_{j_0}$, $D^{s+1}{\bf u}(x_1)$ and $D^sp(x_1)$ are similarly solved. Thus, combining \eqref{estinduction}, \eqref{est-Dmathu}, \eqref{piecebreveu}, and Assumption \ref{assump},  we derive
\begin{align*}
[D^{s+1}{\bf u}]_{\delta_{\mu};B_{1-\varepsilon}\cap\overline{\cD}_{j_0}}+[D^{s}p]_{\delta_{\mu};B_{1-\varepsilon}\cap\overline{\cD}_{j_0}}\leq N\mathcal{C}_4
\end{align*}
for $j_0=1,\ldots,m+1$.
Theorem \ref{Mainthm} with $s\geq2$ follows.\qed

\bibliographystyle{plain}

\end{document}